\documentclass[10pt]{amsart}

\usepackage[margin=1in]{geometry} 
\usepackage{amsmath,amssymb,bm}
\usepackage{amsthm}
\usepackage{graphicx,xcolor}
\usepackage{enumitem}
\usepackage{tikz}
\usetikzlibrary{arrows.meta}
\usepackage{hyperref}

\newcommand{\Z}{\mathbb{Z}}

\newcommand{\R}{\mathbb{R}}

\newcommand{\F}{\mathbb{F}}
\newcommand{\stirling}[2]{\genfrac{[}{]}{0pt}{}{#1}{#2}} 
\newcommand{\BE}{\mathrm{BE}}

\newcommand{\Dirichlet}{\mathrm{Dirichlet}}

\newcommand{\Ic}{\mathcal{I}}
\newcommand{\Comp}{\mathrm{Comp}}
\newcommand{\Dc}{\mathcal{D}}
\newcommand{\LcDc}{\mathcal{LD}}

\newcommand{\mc}{\mathcal{M}}
\newcommand{\Oc}{\mathcal{O}}
\newcommand{\GL}{\mathrm{GL}}

\newcommand{\PF}{\mathrm{PF}}
\newcommand{\IPF}{\mathrm{IPF}}
\newcommand{\mix}{\mathrm{mix}}
\newcommand{\rel}{\mathrm{rel}}
\newcommand{\hist}{\mathrm{hist}}
\newcommand{\dens}{\mathrm{dens}}

\newcommand{\TV}{\mathrm{TV}}

\DeclareMathOperator{\Unif}{Unif}
\renewcommand{\epsilon}{\varepsilon} 
\newcommand{\Address}{{
\bigskip
\footnotesize

\textsc{Department of Mathematics, University of Southern California, Los Angeles, CA, 90089, USA}\par\nopagebreak
\textit{E-mail address}: \texttt{ifeng@usc.edu} \\

\textsc{Department of Mathematics \& Statistics, McMaster University, Hamilton, ON, L8S 4K1, Canada}\par\nopagebreak
\textit{E-mail address}: \texttt{paguyoj@mcmaster.ca}
}}

\def\bal#1\eal{\begin{align*}#1\end{align*}}
 \newcommand{\DrawDyckGrid}[1]{%
  \draw[step=1, gray!25, very thin] (0,0) grid (#1,#1);
  \draw[dashed, gray!70] (0,0)--(#1,#1);
}

\newcommand{\DrawExampleDyckPath}{%
  \draw[line width=1.1pt]
    (0,0)--(0,1)--(0,2)--(1,2)--(2,2)--(2,3)--(3,3)--(3,4)--(3,5)--(4,5)--(5,5);
  \fill (0,0) circle (1.4pt);
  \fill (5,5) circle (1.4pt);
}

\newcommand{\DyckStepLabel}[3]{%
  \node[fill=white, inner sep=1.2pt, font=\scriptsize] at (#1,#2) {#3};
}

\newtheorem{theorem}{Theorem}[section]
\newtheorem{lemma}[theorem]{Lemma}
\newtheorem{proposition}[theorem]{Proposition}
\newtheorem{corollary}[theorem]{Corollary}

\newtheorem{remark}[theorem]{Remark}

\theoremstyle{definition}

\title[Analysis of a twisted Bose--Einstein Markov chain]{Analysis of a twisted Bose--Einstein Markov chain with applications to sampling Catalan structures}
\author{Ivan Z. Feng and J. E. Paguyo}

\subjclass[2020]{60J10, 60C05}
\keywords{Markov chains, mixing times, eigenvalues, spectral radius, twisted Burnside process, Bose--Einstein chain, cut-and-paste process, Dirichlet distribution, parking functions, Dyck paths, triangulations, Catalan structures}

\begin{document}

\begin{abstract}
We analyze a twisted Bose--Einstein Markov chain on $[k]^n$ arising from the Diaconis--Zhong twisted Burnside process. We derive explicit formulas for the transition kernel and stationary distribution of both the original chain and its lumped process, and we identify this chain as a special case of Crane's cut-and-paste process. We establish bounds on the mixing time and determine the second-largest eigenvalue of the original chain and its lumped process. Complementing the fixed-$k$ cutoff theorem of Crane and Lalley, our bounds show that both chains mix in $\Theta(\log n)$ steps when $k$ grows as a fixed positive power of $n$. For the non-twisted Bose--Einstein Markov chain, this resolves a conjecture of Diaconis.
As an application, we study the Burnside processes on parking functions and labeled Dyck paths, which give novel Markov chain Monte Carlo algorithms for sampling an increasing parking function and a Dyck path, respectively, approximately uniformly at random. We show that these chains are rapidly mixing, with mixing times of $\Theta(\log n)$. 
\end{abstract}

\maketitle

\section{Introduction} 

Markov chain Monte Carlo algorithms are a class of algorithms used to sample complex probability distributions and are a mainstay in many scientific fields \cite{Dia08, Dia13}. Letting $X$ be a finite set and $\pi$ a distribution on $X$, the computational problem is to sample from $X$ according to the distribution $\pi$. Markov chain Monte Carlo provides an algorithmic method for solving this problem by running a Markov chain on $X$ whose stationary distribution is $\pi$. The practical uses of these algorithms depend on the mixing time of the Markov chain, which is the number of steps the chain must be run in order to get sufficiently close to its stationary distribution. In practice $X$ is a large set, so an efficient algorithm should have a mixing time that is much smaller than $|X|$. The mixing time analysis of Markov chain Monte Carlo algorithms remains an active research area. 

In this paper, we focus on the Burnside process, a class of Markov chain Monte Carlo algorithms used to sample unlabeled objects, combinatorial structures modulo a group of symmetries. It is a special case of the {\em auxiliary variables} Markov chain \cite{ES88}, a class of unifying algorithms which includes the {\em Swendsen-Wang algorithm} \cite{SW87}, {\em data augmentation} \cite{TW87}, and {\em hit-and-run} \cite{AD07}. These are non-local chains which are able to make big jumps in the state space in a single step. 

\subsection{Burnside Process}

Let $G$ be a finite group acting on a finite set $X$. The group action splits $X$ into disjoint {\em orbits}. Let $O_x = \{gx : g \in G\}$ be the orbit containing $x \in X$. The {\em Burnside process} is a Markov chain introduced by Jerrum \cite{Jer93} as a practical algorithm for sampling an orbit approximately uniformly at random.

Let $G_x = \{g \in G : gx = x\}$ be the {\em stabilizer} of $x \in X$ and let $X_g = \{ y \in X : gy = y\}$ be the {\em fixed set} of $g \in G$. The Burnside process is a reversible Markov chain on $X$ whose transition between states $x$ and $y$ is given by the following two substeps:
\begin{itemize}
\item From $x \in X$, choose $g \in G_x$ uniformly at random, then
\item From $g$, choose $y \in X_g$ uniformly at random.
\end{itemize}
The transition matrix is given by
\begin{equation} \label{BPTransEqn}
K(x,y) = \sum_{g \in G_x \cap G_y} \frac{1}{|G_x|} \frac{1}{|X_g|}
\end{equation}
with stationary distribution 
\begin{equation} \label{BPStationaryEqn}
\pi(x) = \frac{1}{z |O_x|}
\end{equation}
where $z$ is the number of orbits. Since the identity element contributes to every transition, \(K(x,y)\geq 1/(|G_x||X|)>0\) for all \(x,y\in X\), and thus the Burnside process is irreducible and aperiodic.

Let $\{\Oc_k\}_{k=1}^z$ be the disjoint orbits of $X$ under the group action $G$, so that $X = \Oc_1 \cup \Oc_2 \cup \dotsb \cup \Oc_z$. Let $\{X_t\}_{t=0}^\infty$ be the Burnside process on $X$. Then the {\em lumped Burnside process} is the Markov chain, $\{Y_t\}_{t=0}^\infty$, such that $Y_t = a$ if $X_t \in \Oc_a$. It follows that the ordinary Burnside process has a uniform stationary distribution when lumped into orbits.

In \cite{DZ21}, Diaconis and Zhong introduced a generalization of the original Burnside process called the {\em twisted Burnside process}, which gives a large family of Burnside processes with tunable stationary distributions. Again, $G$ is a finite group acting on a finite set $X$. Let $w$ be a positive weight on $G$ and let $W(x) = \sum_{g \in G_x} w(g)$. Let $v$ be a positive weight on $X$ and let $V(g) = \sum_{x \in X_g} v(x)$. The twisted Burnside process is a reversible Markov chain on $X$ whose transition between states $x$ and $y$ is given by the following two substeps:
\begin{itemize}
\item From $x \in X$, choose $g \in G_x$ with probability $\frac{w(g)}{W(x)}$, then
\item From $g$, choose $y \in X_g$ with probability $\frac{v(y)}{V(g)}$.
\end{itemize}
The transition matrix is given by
\begin{equation} \label{eq:general-twisted-Burnside-kernel}
K(x,y) = \frac{v(y)}{W(x)} \sum_{g \in G_x \cap G_y} \frac{w(g)}{V(g)}
\end{equation}
with stationary distribution 
\begin{equation} \label{eq:general-twisted-Burnside-pi}
\pi(x) = \frac{W(x)v(x)}{Z},
\end{equation}
where $Z = \sum_{x\in X}W(x)v(x)=\sum_{g\in G}w(g)V(g)$.
For every real-valued function $f$ on $X$,
\begin{align*}
\langle f,Kf\rangle_\pi
=\frac1Z\sum_{x,y\in X}v(x)v(y)f(x)f(y)
\sum_{g\in G_x\cap G_y}\frac{w(g)}{V(g)}
=\frac1Z\sum_{\substack{g\in G\\V(g)>0}}\frac{w(g)}{V(g)}
\left(\sum_{x\in X_g}v(x)f(x)\right)^2\ge0.
\end{align*}
Thus \(K\) is positive semidefinite on \(L^2(\pi)\), and hence has no negative eigenvalues. If $w$ is constant on each conjugacy class of $G$ and $v$ is constant on each orbit of $X$ (under the action of $G$), then the twisted Burnside process can be lumped onto the orbits of $X$. The lumped chain is a reversible Markov chain called the {\em lumped twisted Burnside process}, with stationary distribution $\bar{\pi}(O_x) \propto \frac{W(x)v(x)}{|G_x|}$; its kernel also has no negative eigenvalues. The special case where $w\equiv v\equiv1$ recovers the original, non-twisted, Burnside process. 

Jerrum showed that for certain choices of the group $G$, the Burnside process is {\em rapidly mixing} \cite{Jer93, Jer95}, meaning the mixing time is upper bounded by a polynomial in its input size. Subsequently, Goldberg and Jerrum showed that it is not rapidly mixing in general \cite{GJ02}. The computational properties and mixing time of the Burnside process have been studied in many special cases: Bose--Einstein configurations \cite{AF02, Dia05, DZ21, DLR25a+, DLR25b+}, conjugacy classes of centralizer Abelian groups \cite{Rah20}, Polya trees \cite{BD26}, integer partitions \cite{DH25}, contingency tables \cite{Dit19, DH25}, the flag variety of $\GL_n(\F_q)$ \cite{DM25+}, Sylow $p$-double cosets \cite{How25+}, set partitions \cite{Pag26}, conjugacy classes of unitriangular groups \cite{DZ26}, and matchings and coagulations of integer partitions \cite{Dal26}. In computer science, the Burnside process has been applied as a subroutine to an approximate lifted probabilistic inference algorithm called the {\em orbit-jump Markov chain Monte Carlo} \cite{HMV20}. More recently, the {\em dual Burnside process}, obtained by interchanging the roles of group elements and states, was systematically studied by the first author in \cite{Fen25}.

The mixing time analysis for the Burnside process, in full generality, remains an open problem. Among the available model-independent results, Chen~\cite[Propositions~10 and~11]{Che06} proved worst-case total-variation decay bounds of $(1-|G|^{-1})^t$ for the original Burnside process and $(1-|X|^{-1})^t$ for the lumped process. These bounds are usually not strong enough to prove rapid mixing, and so the Burnside process is typically studied on a case-by-case basis using various methods such as coupling, minorization, spectral theory, and geometric bounds.

\subsection{The Twisted Bose--Einstein Markov Chain and Main Results}\label{MainResults}

Let $n, k \ge 1$ be integers, and let $\theta > 0$ be a real number. Let $[k] = \{1,\ldots,k\}$. Let the symmetric group $S_n$ act on the set $[k]^n$ by permuting coordinates. For all $\sigma \in S_n$, let 
\bal
w(\sigma) = \theta^{\ell(\sigma)},
\eal
where $\ell(\sigma)$ is the number of cycles of $\sigma$. For all $x = (x_1, \ldots, x_n), y = (y_1, \ldots, y_n) \in [k]^n$, define
\bal
i_a(x) := |\{r \in [n] : x_r = a\}| \quad \text{and} \quad m_{a,b}(x,y) := |\{r \in [n] : x_r = a, \ y_r = b\}|.
\eal
Let $(x)_{(n)} = x(x+1)\dotsb (x+n-1)$ be the {\em rising factorial}, with the convention $(x)_{(0)} = 1$. 

The {\em twisted Bose--Einstein Markov chain on $[k]^n$} is a Markov chain on $[k]^n$ whose transition between states $x$ and $y$ is given by the following two substeps: 
\begin{itemize}
\item From $x \in [k]^n$, choose $\sigma \in G_x$ with probability $\frac{\theta^{\ell(\sigma)}}{W(x)}$, then
\item From $\sigma$, choose $y \in ([k]^n)_\sigma$ uniformly from the fixed set of $\sigma$.
\end{itemize}
Let $K$ and $\pi$ be the transition kernel and stationary distribution, respectively, of the twisted Bose-Einstein Markov chain. In Section~\ref{SectionTransitionKernels}, we show that 
\bal
 K(x,y) =\prod_{a=1}^k
 \frac{\displaystyle\prod_{b=1}^k
 (\theta/k)_{(m_{a,b}(x,y))}}
 {(\theta)_{(i_a(x))}}
\quad \text{and} \quad
 \pi(x)
 =\frac{\displaystyle\prod_{a=1}^k
 (\theta)_{(i_a(x))}}{(k\theta)_{(n)}}.
\eal

The same formulas identify this chain with Crane's symmetric Dirichlet cut-and-paste process \cite[Example~1.7, Equation~(1.8)]{Crane14}; see Remark \ref{rem:cut-and-paste}. The orbits of this chain are indexed by the set of {\em Bose-Einstein configurations}, denoted by $\Comp_{n,k}$, which are the possible configurations of dropping $n$ unlabeled balls into $k$ boxes. The twisted Bose--Einstein Markov chain, when lumped onto orbits, is called the {\em lumped twisted Bose--Einstein Markov chain on $\Comp_{n,k}$}.  In Section \ref{SectionLumpedBE}, we derive the explicit transition kernel and stationary distribution of this lumped chain. 

The special case where $\theta=1$ corresponds to the {\em Bose-Einstein Markov chain on $[k]^n$} \cite{Dia05}. This is a Markov chain whose transition between states $x,y \in [k]^n$ can be described explicitly as:
\begin{itemize}
\item From \(x\in[k]^n\), identify the set of coordinates \(I_j\) with common value \(j\) for all \(1\le j\le k\). Choose independent uniformly random permutations \(w_1,w_2,\ldots,w_k\) of these sets.
\item Break each \(w_i\) into cycles and label the coordinates of each cycle with an independent uniform choice in \([k]\). Let the final configuration be \(y\).
\end{itemize}
In this case, the lumped stationary distribution is uniform on the set of Bose--Einstein configurations. Thus the Bose--Einstein Markov chain gives a Markov chain Monte Carlo algorithm for sampling a Bose-Einstein configuration approximately uniformly at random. 

Let $P$ be a reversible Markov chain on $X$ with stationary distribution $\pi$. Define
\[
 d_P(t) = \max_{x \in X}\|P^t(x,\cdot)-\pi\|_{\TV} \quad \text{and} \quad 
 t_{\mix, P}(\epsilon):=\min\left\{t\ge0: d_P(t) \le\epsilon \right\}.
\]
These are the {\em distance function} and the {\em mixing time} of $P$. For the twisted Bose–Einstein kernel \(K\), write \(d(t):=d_K(t)\) and \(t_{\mix}(\epsilon):=t_{\mix,K}(\epsilon)\). Let $\lambda_2(P)$ be the second largest eigenvalue of $P$. Let $\Gamma(x)$ be the {\em gamma function} and let $\psi(z) = \frac{\Gamma'(z)}{\Gamma(z)}$ be the {\em digamma function}.

Our main result determines the second largest eigenvalue of the twisted Bose--Einstein Markov chain on $[k]^n$ and establishes upper and lower bounds on its mixing time. 

\begin{theorem}\label{BEbound}
Fix integers $k,n \ge 2$. Let $K$ be the transition kernel of the twisted Bose-Einstein Markov chain on $[k]^n$ and let $\pi$ be its stationary distribution. Then $\lambda_2 = \lambda_2(K) = \frac{k-1}{k(\theta + 1)}$, and for all $t \geq 1$, 
\bal
 \max\left\{
 \frac{\max\{\theta + 1,\theta(k-1)\}}{k\theta + 1}\lambda_2^t,\,
 1-\frac{(\psi(\theta+1)-\psi(1))t+\log(2 + \theta^{-1})}{\log(\min\{n,k\})}
 \right\} 
 \le d_K(t) 
 \le
 \frac n2\sqrt{2k}\,\lambda_2^{t/2}.
\eal
Consequently, for every fixed $\theta>0$ and $0<\epsilon<1$, if $k=k_n=n^{\beta+o(1)}$, where $\beta>0$ is a fixed constant and $o(1) \to 0$ as $n \to \infty$, then
\[
 t_{\mix,K}(\epsilon)
 =\Theta_{\theta,\beta,\epsilon}(\log n), 
\]
where $\Theta_{\theta,\beta,\epsilon}(\cdot)$ means that the constant depends on $\theta,\beta,\epsilon$. 
\end{theorem}

The lumped twisted Bose--Einstein Markov chain has the same second-largest eigenvalue, upper bound, and entropy lower bound. In particular, it has the same logarithmic mixing order when $k=n^{\beta+o(1)}$.

Recall that the special case where $\theta = 1$ corresponds to the Bose--Einstein Markov chain. In \cite{AF02}, Aldous used a coupling argument to show that $d_K(t)\le n(1-1/k)^t$, which gives a mixing-time upper bound of $O(k \log n)$ uniformly in the starting state $x \in [k]^n$. Subsequently in \cite{Dia05}, Diaconis used minorization to show that when the chain is started from a constant state, say $(1,\ldots,1)$, the mixing time is upper bounded by $O(k!)$, so that for fixed $k$ the mixing time is independent of $n$. Finally, sharp mixing time bounds have been obtained for the special case where $k = 2$ \cite{DZ21, DLR25a+, DLR25b+}. 

Observe that Aldous' bound is better for $k$ large, while Diaconis' bound is better for fixed $k$ or $k$ growing very slowly with $n$. Diaconis conjectured~\cite[Section~4, remarks after Theorem~2]{Dia05} that the mixing time of the Bose--Einstein Markov chain is of order $\log n$ when $k$ is of order $n$. When $k = \Theta(n)$ and $\theta=1$, then Theorem~\ref{BEbound} implies that the mixing time of the Bose--Einstein Markov chain is $\Theta(\log n)$, thus resolving Diaconis' conjecture. Moreover, taking the trivial bound $d_K(t)\le1$ into account, our upper bound is uniformly no worse than Aldous' bound. Finally, Diaconis' bound applies only when the chain is started from a constant state and is strongest for fixed $k$, whereas our bound is uniform over the starting state.

By the correspondence between the twisted Bose--Einstein chain and the symmetric Dirichlet cut-and-paste process, for every fixed $k\ge2$ and $\theta>0$, Crane and Lalley's Example~5.8 and Theorem~5.7 \cite{CL12} show that this chain has cutoff at order $\log n$. In particular, for $k=2$ and $\theta=1$, a direct calculation gives the corresponding Lyapunov factor $1/4$, so, for every fixed $0<\epsilon<1$,
\[
 t_{\mix}(\epsilon)\sim\frac{\log n}{4\log2}=\frac14\log_2n,
 \qquad n\to\infty.
\]
Theorem~\ref{BEbound} complements their result with finite-$n$ bounds valid when $k$ grows with $n$, including matching logarithmic-order bounds when $k=n^{\beta+o(1)}$, a regime not covered by their fixed-dimensional theorem. 

\subsection{Application to Parking Functions and Dyck Paths}

Let $n \geq 1$ be an integer. A sequence of $n$ positive integers $x = (x_1, \ldots, x_n) \in [n]^n$ is a {\em parking function} of length $n$ if and only if $x_{(i)} \leq i$ for all $i \in [n]$, where $x_{(1)} \leq \dotsb \leq x_{(n)}$ is the {\em weakly increasing rearrangement} of $x$. Equivalently, $x$ is a parking function if and only if $|\{k\in[n]:x_k \leq i\}| \geq i$ for all $i \in [n]$. This implies that parking functions are invariant under permutations of the coordinates.

Let $\PF_n$ be the set of parking functions of length $n$, which has cardinality $|\PF_n| = (n+1)^{n-1}$. An elegant, unpublished proof of this result using a circular symmetry argument was given by Pollak and recounted in \cite{FR74}. Parking functions were introduced by Konheim and Weiss \cite{KW66} in their study of the hash storage structure and they continue to be a topic of active research, with wide applications to combinatorics, probability, and computer science. We refer the reader to the survey paper \cite{Yan15} and the references therein.

An {\em increasing parking function} of length $n$ is a parking function, $x \in \PF_n$, of length $n$ whose entries are in weakly increasing order, $x_1 \leq \dots \leq x_n$. Let $\IPF_n$ be the set of increasing parking functions of length $n$. It is known that $|\IPF_n| = C_n$, the $n$th {\em Catalan number}, given by $C_n = \frac{1}{n+1}\binom{2n}{n}$. The proof follows via a bijection between $\IPF_n$ and the set of {\em Dyck paths} of length $2n$, which is counted by $C_n$. Consequently, increasing parking functions are in bijection with a wide range of other {\em Catalan objects} \cite{Sta15}.

Let $S_n$ act on $\PF_n$ by permuting coordinates. That is, if $x = (x_1,\dots,x_n) \in \PF_n$ and $\sigma\in S_n$, then \(\sigma x=(x_{\sigma^{-1}(1)},\ldots,x_{\sigma^{-1}(n)})\). The {\em Burnside process on parking functions} (henceforth, the {\em Burnside process on $\PF_n$}) is a Markov chain whose transition between states $x,y \in \PF_n$ is described by the following two-step procedure:
\begin{itemize}
\item From $x \in \PF_n$, pick $\sigma \in G_x$ uniformly at random.
\item From $\sigma$, pick $y \in (\PF_n)_\sigma$ uniformly at random.
\end{itemize}

Let $x \in \PF_n$. Let $\bar{x}$ be its weakly increasing rearrangement and observe that $\bar{x} \in \IPF_n$. Since parking functions are invariant under permutations of their coordinates, two parking functions $x,y \in \PF_n$ lie in the same orbit if and only if $\bar{x} = \bar{y}$. Thus the orbits are indexed by increasing parking functions of length $n$. If $O_x$ is the orbit containing $x$, then we say that $\bar{x}$ is the increasing parking function {\em corresponding to} the orbit $O_x$. When lumped into orbits, the Burnside process on $\PF_n$ defines a Markov chain on the state space $\IPF_n$ with uniform stationary distribution. We refer to this lumped chain as the {\em lumped Burnside process on increasing parking functions} (henceforth, the {\em lumped Burnside process on $\IPF_n$}).
Thus the Burnside process on $\PF_n$ gives a novel Markov chain Monte Carlo algorithm for sampling an increasing parking function from $\IPF_n$ approximately uniformly at random. 

As an application of our Theorem \ref{BEbound}, we show that the Burnside process on $\PF_n$ has a mixing time of order $\log n$.  

\begin{theorem} \label{mainthm}
Fix $n \geq 2$. 
Let $K_{\PF}$ be the transition kernel of the Burnside process on $\PF_n$ and let $\pi_{\PF}$ be its stationary distribution. 
Then $\lambda_{2,\PF} = \lambda_2(K_{\PF}) = \frac{n}{2(n+1)}$, and for all $t\ge1$, 
\begin{equation}
 \max\left\{
  \frac{n}{n+2}\lambda_{2,\PF}^t,\,
  1-\frac{t+\log 3}{\log n}
 \right\}
 \le d_{K_{\PF}}(t)
 \le n\sqrt{\frac{n+1}{2}} \lambda_{2,\PF}^{t/2}.
 \label{eq:intro-PF-bounds}
\end{equation}
Consequently, for every fixed $0<\epsilon<1$,
\[
 t_{\mix,K_{\PF}}(\epsilon)=\Theta_\epsilon(\log n).
\]
\end{theorem}

Let $\Dc_n$ be the set of Dyck paths of length $2n$. It is well known that $|\Dc_n| = C_n$. Moreover, recall that there exists a bijection between $\IPF_n$ and $\Dc_n$, which we now describe. This exposition follows \cite{ALW16}. A Dyck path of length $2n$ is a lattice path in $\Z^2$ from $(0,0)$ to $(n,n)$ that weakly stays above the diagonal line $y = x$. It can also be represented by a sequence in $\{N, E\}^{2n}$, where $N$ represents an up-step, $E$ represents a right-step, and such that every initial subsequence has at least as many $N$'s as $E$'s. Given an increasing parking function $u = (u_1, \ldots, u_n) \in \IPF_n$, let $r_i$ be the number of times that $i$ occurs in $u$. Then $u$ is associated with the Dyck path 
\bal
\underbrace{N \dotsb N}_{r_1} E \underbrace{N \dotsb N}_{r_2} E \dotsb \underbrace{N \dotsb N}_{r_n} E.
\eal
See Figure~\ref{fig:dyck-bijections}\textup{(a)} for an example.

Generalizing this, we also describe a bijection between labeled Dyck paths and parking functions. A {\em labeled Dyck path} of length $2n$ is a Dyck path of length $2n$ where each up-step is labeled by a distinct integer $k \in [n]$ such that the labels of consecutive up-steps (which we refer to as a {\em vertical run}) are in strictly increasing order. Let $\LcDc_n$ be the set of labeled Dyck paths of length $2n$. It was shown in \cite{GH96} that $\LcDc_n$ is in bijection with the set $\PF_n$. Given a parking function $x = (x_1, \ldots, x_n) \in \PF_n$, first draw the Dyck path corresponding to its weakly increasing rearrangement $\bar{x} \in \IPF_n$. The $i$th vertical run has length $r_i$ (possibly $0$) corresponding to the number of occurrences of $i$ in $x$. If $r_i = k$ with $i = x_{j_1} = \dotsb = x_{j_k}$, then we label the $i$th vertical run by the increasing set of indices $j_1 < \ldots < j_k$. See Figure~\ref{fig:dyck-bijections}\textup{(b)} for an example.
\begin{figure}[htbp]
\centering
\begin{tikzpicture}[x=.72cm,y=.72cm, >=Latex]

\begin{scope}
  \DrawDyckGrid{5}
  \DrawExampleDyckPath

  \node[font=\small] at (2.5,5.75)
    {\textup{(a)} \(u=(1,1,3,4,4)\in\IPF_5\)};

  \node[align=center, font=\scriptsize] at (2.5,-.85)
    {\(\hist(u)=(2,0,1,2,0)\)\\[2pt]
     \(N^2E\,E\,NE\,N^2E\,E\)};
\end{scope}

\begin{scope}[shift={(8.2,0)}]
  \DrawDyckGrid{5}
  \DrawExampleDyckPath

  \DyckStepLabel{-.28}{.5}{\(2\)}
  \DyckStepLabel{-.28}{1.5}{\(5\)}
  \DyckStepLabel{1.75}{2.5}{\(3\)}
  \DyckStepLabel{3.28}{3.5}{\(1\)}
  \DyckStepLabel{3.28}{4.5}{\(4\)}

  \node[font=\small] at (2.5,5.75)
    {\textup{(b)} \(x=(4,1,3,4,1)\in\PF_5\)};

 \node[align=center, font=\scriptsize] at (2.5,-.85)
  {\(\bar x=(1,1,3,4,4)\)\\[2pt]
   value \(1\): \(\{2,5\}\), value \(3\): \(\{3\}\), value \(4\): \(\{1,4\}\)};
\end{scope}

\end{tikzpicture}
\caption{
\textup{(a)} The Dyck path (given by the sequence $N^2E\,E\,NE\,N^2E\,E$) corresponding to the increasing parking function \(u=(1,1,3,4,4)\).
\textup{(b)} The labeled Dyck path corresponding to the parking function
\(x=(4,1,3,4,1)\). 
}
\label{fig:dyck-bijections}
\end{figure}
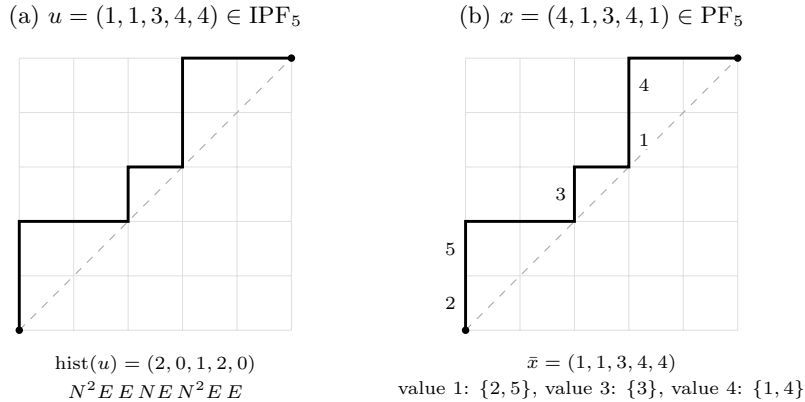

Let $X = \LcDc_n$ and let $G = S_n$ act on $\LcDc_n$ as follows. Given \(\sigma\in S_n\), replace the label \(i\) by \(\sigma(i)\), and then rearrange the labels along each vertical run so that they are in increasing order. See Figure~\ref{fig:labeled-dyck-action}. It is clear that under this group action, $\LcDc_n$ breaks up into $C_n$ many orbits, indexed by $\Dc_n$. 

\begin{figure}[htbp]
\centering
\begin{tikzpicture}[x=.72cm,y=.72cm, >=Latex]

\begin{scope}
  \DrawDyckGrid{5}
  \DrawExampleDyckPath

  \DyckStepLabel{-.28}{.5}{\(2\)}
  \DyckStepLabel{-.28}{1.5}{\(5\)}
  \DyckStepLabel{1.75}{2.5}{\(3\)}
  \DyckStepLabel{3.28}{3.5}{\(1\)}
  \DyckStepLabel{3.28}{4.5}{\(4\)}

  \node[font=\small] at (2.5,5.75)
    {\(x=(4,1,3,4,1)\)};

  \node[align=center, font=\scriptsize] at (2.5,-.85)
  {value \(1\): \(\{2,5\}\), value \(3\): \(\{3\}\),\\
   value \(4\): \(\{1,4\}\)};
\end{scope}

\node[align=center, font=\small] at (6.8,3.1)
  {\(\sigma=(1\,5\,3)(2\,4)\)};

\draw[-{Latex[length=3mm]}, line width=.9pt] (5.6,2.35)--(8.1,2.35);

\begin{scope}[shift={(9.1,0)}]
  \DrawDyckGrid{5}
  \DrawExampleDyckPath

  \DyckStepLabel{-.28}{.5}{\(3\)}
  \DyckStepLabel{-.28}{1.5}{\(4\)}
  \DyckStepLabel{1.75}{2.5}{\(1\)}
  \DyckStepLabel{3.28}{3.5}{\(2\)}
  \DyckStepLabel{3.28}{4.5}{\(5\)}

  \node[font=\small] at (2.5,5.75)
    {\(\sigma\cdot x=(3,4,1,1,4)\)};

 \node[align=center, font=\scriptsize] at (2.5,-.85)
  {value \(1\): \(\{3,4\}\), value \(3\): \(\{1\}\),\\
   value \(4\): \(\{2,5\}\)};
\end{scope}

\end{tikzpicture}
\caption{
Example of the group action of \(S_5\) on labeled Dyck paths $\LcDc_5$.
Here \(\sigma=(1\,5\,3)(2\,4)\).
}
\label{fig:labeled-dyck-action}
\end{figure}
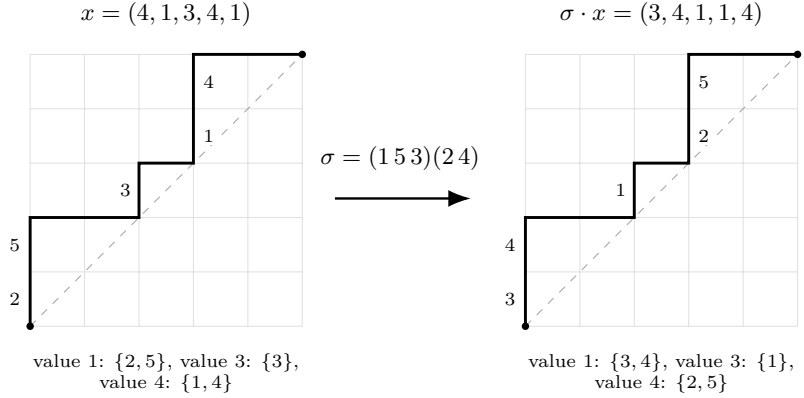

Let $G_x$ be the stabilizer of $x \in \LcDc_n$ and let $X_\sigma$ be the fixed set of $\sigma$ under this group action of $S_n$ on $\LcDc_n$. The {\em Burnside process on labeled Dyck paths} (henceforth, the {\em Burnside process on $\LcDc_n$}) is a Markov chain whose transition between states $x,y \in \LcDc_n$ is described by the following two-step procedure:
\begin{itemize}
\item From $x \in \LcDc_n$, pick $\sigma \in G_x$ uniformly at random. 
\item From $\sigma$, pick $y \in X_\sigma$ uniformly at random. 
\end{itemize}
When lumped into orbits, the Burnside process on $\LcDc_n$ defines a Markov chain on $\Dc_n$ with uniform stationary distribution. We call this lumped chain the {\em lumped Burnside process on Dyck paths} (henceforth, the {\em lumped Burnside process on $\Dc_n$}). 

Let $K_{\LcDc}$ denote the transition kernel of the Burnside process on $\LcDc_n$. The Burnside process on $\LcDc_n$ gives a novel Markov chain Monte Carlo algorithm for sampling a Dyck path from $\Dc_n$ approximately uniformly at random. We remark that the bijection \(F:\PF_n\to\LcDc_n\) between parking functions and labeled Dyck paths described above is \(S_n\)-equivariant; that is,
\[
F(\sigma x)=\sigma\cdot F(x)
\]
for all \(\sigma\in S_n\) and \(x\in\PF_n\). Thus the coordinate action on \(\PF_n\) corresponds exactly to the label-permutation action on \(\LcDc_n\). This equivariance preserves stabilizers and fixed sets: for every \(x\in\PF_n\) and \(\sigma\in S_n\),
\[
G_{F(x)}=G_x,
\qquad
F\bigl((\PF_n)_\sigma\bigr)=(\LcDc_n)_\sigma.
\]
Hence
\[
K_{\LcDc}(F(x),F(y))=K_{\PF}(x,y),
\]
so the Burnside processes on \(\PF_n\) and \(\LcDc_n\) are identical after identifying \(x\) with \(F(x)\). 
Therefore Theorem~\ref{mainthm} also holds when $K_{\PF}$ is replaced by $K_{\LcDc}$. 
To the best of our knowledge, among Markov chain Monte Carlo algorithms for sampling Dyck paths, this is the fastest proved mixing bound when measured in Markov-chain transitions.

\subsection{Outline}

The paper is organized as follows. In Section \ref{Sec:MCMCCatalan}, we survey previous work on Markov chains on Catalan structures and discuss how our Burnside processes on parking functions and Dyck paths fit within this framework. In Section \ref{SectionPreliminaries}, we collect the notation, definitions, and auxiliary results that will be used throughout the paper. In Section \ref{SectionTransitionKernels}, we begin with the Diaconis--Zhong twisted Burnside specialization, use the fixed-set and stabilizer calculations on $[k]^n$ to derive the labeled kernel $K$ and its stationary distribution $\pi$, identify the labeled chain with Crane's Dirichlet cut-and-paste chain, and then derive the general orbit-lumped kernel and its stationary distribution. In Section \ref{SectionComparison}, we recover the parking-function and increasing-parking-function kernels, after which Pollak's cyclic shifts identify the ordinary Bose--Einstein chain with the parking-function chain together with an independent uniform cyclic shift. Section~\ref{SectionMixingSpectral} uses the Dirichlet representation to identify the second-largest eigenvalue and prove the two-parameter upper and entropy lower bounds of the twisted Bose--Einstein Markov chain, then transfers these conclusions to parking functions and their orbit chains. The bijection between parking functions and labeled Dyck paths respects the $S_n$-actions and gives the labeled-Dyck conclusions stated above.

\subsection*{Acknowledgments}

We thank Jason Fulman for many helpful discussions and pointers to references.

\section{Markov Chains on Catalan Structures} \label{Sec:MCMCCatalan}

In this section, we survey some previous work on Markov chain Monte Carlo algorithms for approximately uniformly sampling Catalan structures and discuss how our Burnside processes on parking functions and Dyck paths fit within this context. 

Many classes of combinatorial objects are enumerated by the Catalan numbers $C_n$, including increasing parking functions of length $n$, Dyck paths of length $2n$, triangulations of a convex $(n+2)$-gon, noncrossing partitions of $[n]$, binary trees with $n$ vertices, and plane trees with $n+1$ vertices. These objects are commonly referred to as {\em Catalan structures}. 

After an $O(n^2)$-time preprocessing step, many standard Catalan structures can be sampled exactly and uniformly in $O(n)$ time per sample~\cite{AB18}. Moreover, due to the many bijections between various Catalan structures, a sampling algorithm for one Catalan structure allows for the sampling of others. On the other hand, analyzing the convergence rates of Markov chain Monte Carlo algorithms for approximately uniformly sampling a Catalan structure is a more interesting problem. In the sequel, we discuss two such Markov chains on Catalan structures; for many more examples, see Cohen's wonderful thesis \cite{Coh16}. The motivation for this line of study comes from a problem posed by Aldous \cite{Ald03}, who conjectured that the triangulation walk, a Markov chain on triangulations of a convex polygon, should have relaxation time of order $n^{3/2}$. This remains an open problem. Markov chains on Catalan structures with nonuniform distribution have also been studied \cite{MR00}. However, due to the lack of progress on many of the chains listed above, the mixing time analysis of Markov chains on Catalan structures is a challenging problem, even in the case of uniform stationary distribution.

\subsection{Triangulation Walk}

A {\em triangulation} of a convex $(n+2)$-gon is a set of $n-1$ non-crossing diagonals which partition the interior of the polygon into $n$ triangles. Let $T_n$ be the set of all triangulations of a convex $(n+2)$-gon, so that $|T_n|$ = $C_n$. Consider the graph $G_n$ with vertex set $V = T_n$, such that two triangulations are connected by an edge if one can be obtained from the other by a diagonal flip \cite{STT88}. Note that each diagonal lies in a unique quadrilateral. A {\em diagonal flip} replaces that diagonal with the other diagonal of the same quadrilateral. 

For $n\ge2$, the {\em triangulation walk} is the $1/2$-lazy random walk on $G_n$: at each step it stays put with probability $1/2$, and otherwise moves to a uniformly chosen neighboring triangulation. It has uniform stationary distribution and gives a Markov chain Monte Carlo algorithm for approximately uniformly sampling a triangulation.

Molloy, Reed, and Steiger \cite{MRS98} initiated the study of the mixing time of the triangulation walk and showed that it is rapidly mixing. In particular, they showed that the mixing time is lower bounded by $\Omega(n^{3/2})$ and upper bounded by $O(n^{25})$. Subsequently, McShine and Tetali \cite{MT98} obtained an improved mixing time upper bound of $O(n^5 \log n)$. More recently, Eppstein and Frishberg \cite{EF23} further improved this upper bound to $O(n^3 \log^3 n)$, the first improvement in 25 years. Even more recently, this was further improved to $\widetilde{O}(n^2)$ by Alev et al. \cite{AFST26}, where $\widetilde{O}(\cdot)$ denotes a bound that hides polylogarithmic factors in $n$.

\subsection{Dyck Random Transpositions Walk}

A {\em Catalan string} of order $n$ is a sequence $x \in \{\pm 1\}^{2n}$ of $n$ many $1$'s and $n$ many $-1$'s with nonnegative partial sums $\sum_{k=1}^j x_k \geq 0$ for all $j \in [2n]$. It is clear that the set of Catalan strings of order $n$ is in bijection with the set of Dyck paths of length $2n$. Thus Catalan strings can be visualized as a Dyck path. 

The {\em Dyck random transpositions walk} is a Markov chain on Catalan strings (equivalently, Dyck paths) defined as follows. Given a Catalan string $x$, pick two indices $i,j \in [2n]$ uniformly at random. Let $\tau_{i,j} = (i \, j) \in S_{2n}$ be the transposition which transposes $i$ and $j$. If $x' = \tau_{i,j}(x)$ yields a valid Catalan string, the chain moves to state $x'$. Otherwise, $x'$ is an invalid Catalan string and the chain remains at state $x$; we say that the move is {\em censored}. Restricting this chain to only adjacent transpositions yields the {\em Dyck adjacent transpositions walk}. One can also define similar walks from any symmetric set of generators of $S_{2n}$. 

On the other hand, removing the censoring step yields the {\em uncensored random transpositions walk on binary strings}, which are strings with an equal number of $1$'s and $-1$'s. This chain is equivalent to the Bernoulli-Laplace model, which is known to have a mixing time of $O(n\log n)$ \cite{DS87}. Restricting only to adjacent transpositions gives the {\em uncensored random adjacent transpositions walk on binary strings}, a chain studied by Wilson \cite{Wil04} who proved a $\Theta(n^3 \log n)$ mixing time. 

Cohen \cite{Coh16} notes that the censoring mechanism of the Dyck random transpositions walk favors moves which are more ``local,'' meaning that the more the Dyck path is changed by a single transposition, the more likely the resulting Dyck path is invalid and the move censored. Cohen, Tetali, and Yeliussizov \cite{CTY15} proved that the mixing time of the Dyck random transpositions walk is upper bounded by $O(n^2 \log n)$. Cohen \cite{Coh16} conjectured that the correct order should be $\Theta(n \log n)$, and this remains an open problem.   

\subsection{Burnside Process on Catalan Structures}

Consider a Burnside process whose number of orbits is a Catalan number. We refer to this chain as a {\em Burnside process on a Catalan structure}. By construction, a Burnside process on a Catalan structure gives a novel Markov chain Monte Carlo algorithm for approximately uniformly sampling a Catalan object. Two examples are the Burnside processes on $\PF_n$ and on $\LcDc_n$, which give algorithms for approximately uniformly sampling an increasing parking function and a Dyck path, respectively.

Recall that the Dyck random transpositions walk favors moves which are more ``local,'' similar to the random transpositions shuffle on $n$ cards \cite{DS81}. On the other hand, our Burnside process on $\LcDc_n$ is able to make large jumps within the state space in a single step, akin to the riffle shuffle on $n$ cards \cite{BD92}. The available upper bound improves from $O(n^2\log n)$ transitions for the Dyck random transpositions walk to $O(\log n)$ transitions for the Burnside process on labeled Dyck paths.

On the other hand, we have not found a natural triangulation-native Burnside process distinct from the one obtained by transporting our labeled-Dyck-path process through a standard bijection. This would give a novel Markov chain Monte Carlo algorithm for approximately uniformly sampling a triangulation. We leave this as an open problem. However, using the known bijections between Dyck paths and triangulations, we can apply our Burnside process on $\LcDc_n$ as a subroutine for an approximate triangulation sampler with an $O_\epsilon(\log n)$-step bound as follows:
\begin{itemize}
\item Starting at $x \in \LcDc_n$, run the Burnside process on $\LcDc_n$ for $O_\epsilon(\log n)$ Markov-chain steps. Let $y \in \LcDc_n$ be the resulting final state. Return the Dyck path, $v \in \Dc_n$, corresponding to the orbit of $y$; from an arbitrary labeled starting state, its law is within total-variation distance at most $\epsilon$ of uniform on $\Dc_n$.
\item From $v$, use any of the known bijections between $\Dc_n$ and $T_n$ to obtain the corresponding triangulation $\tau \in T_n$. For many bijections this takes $O(n)$ time.
\end{itemize}

Following Diaconis--Howes \cite{DH25}, we distinguish Markov-chain steps from per-step implementation cost: the $O_\epsilon(\log n)$ bound counts transitions of the Burnside chain, while each transition requires carrying out two substeps, sampling $\sigma\in G_x$ and then sampling $y\in X_\sigma$. Therefore this approximate sampling algorithm for triangulations uses $O_\epsilon(\log n)$ Burnside transitions, together with their implementation cost and the $O(n)$-time bijective conversion from the final Dyck path to a triangulation. At the level of Markov-chain transitions, this logarithmic guarantee contrasts with the current $\widetilde O(n^2)$ mixing-time upper bound for the triangulation walk; the per-step costs of the two algorithms are different and are not included in this comparison.

Finally, it would be interesting to construct and analyze the mixing time of the Burnside process on Catalan structures in full generality. The main challenge is in constructing the labeled versions of these Catalan structures, which is the starting point to constructing such Burnside processes. We plan to continue this line of study in future work.

\section{Preliminaries} \label{SectionPreliminaries}

\subsection{Notation and Definitions} \label{Sec:Notation}

For a permutation $\sigma$, let $\ell(\sigma)$ denote the total number of cycles of $\sigma$.

Let $z>0$ and let $n \geq 0$ be a nonnegative integer. Let $\Gamma(z)$ be the {\em gamma function}. The {\em rising factorial} is defined as $(z)_{(n)} = z(z+1)\dotsb (z+n-1)$. The {\em unsigned Stirling numbers of the first kind} count the number of permutations of size $n$ with $k$ cycles, given by $\stirling{n}{k} = |\{\sigma \in S_n : \ell(\sigma) = k\}|$. These quantities are related through the following identities. We use the empty-product convention $(z)_{(0)}=1$.
\begin{equation} \label{eq:stirl1}
(z)_{(n)} = \frac{\Gamma(z + n)}{\Gamma(z)} = \sum_{k=0}^n \stirling{n}{k} z^k = \sum_{\tau\in S_n} z^{\ell(\tau)}.
\end{equation}

Let $n,k\ge1$ and let $x = (x_1, \ldots, x_n), y = (y_1, \ldots, y_n) \in [k]^n$. We define the following quantities. We use $n$ for the word length throughout; the parking-function application takes $k=n+1$.
\begin{itemize}
\item For each $a \in [k]$, the {\em $a$-index set} of $x$ is the set of coordinates equal to $a$,
\bal
I_a(x) = \{r \in [n] : x_r = a\}.
\eal
Let $i_a(x) = |I_a(x)|$, and observe that $\sum_{a=1}^k i_a(x) = n$.
\item The {\em value partition} of $x$ is the set
\bal
\Ic(x) = \{I_a(x) : 1 \leq a \leq k, I_a(x) \neq \emptyset\}.
\eal
\item The {\em histogram} of $x$ is the vector
\bal
\hist(x) = (i_1(x), \ldots, i_k(x)).
\eal
\item The {\em weakly increasing rearrangement} of $x$ is the vector
\[
\bar x=
\bigl(
\underbrace{1,\dots,1}_{i_1(x)},
\underbrace{2,\dots,2}_{i_2(x)},
\dots,
\underbrace{k,\dots,k}_{i_k(x)}
\bigr) \in [k]^n.
\]
\item For all $a,b \in [k]$, the {\em $a,b$-intersection set} of $x$ and $y$ is the set
\bal
B_{a,b}(x,y) = I_a(x) \cap I_b(y) = \{r \in [n] : x_r = a, y_r = b\}.
\eal
Let $m_{a,b}(x,y) = |B_{a,b}(x,y)|$.
\item The {\em density} of $x$ and $y$ is the matrix given by
\bal
\dens(x,y) = (m_{a,b}(x,y))_{a,b=1}^k.
\eal
\item The set of {\em weak compositions of \(n\) into \(k\) parts} is
\bal
\Comp_{n,k}
=
\left\{
(c_1,\dots,c_k)\in\Z_{\ge0}^k:
\sum_{b=1}^k c_b=n
\right\}, 
\eal
with $|\Comp_{n,k}|=\binom{n+k-1}{k-1}$.
\item For $\bm{i},\bm{j}\in\Comp_{n,k}$, the space of {\em $(\bm{i},\bm{j})$-contingency tables} is
\bal
\mc(\bm{i},\bm{j})
=
\left\{
(m_{a,b})_{a,b=1}^k\in\Z_{\ge0}^{k\times k}:
\sum_{b=1}^k m_{a,b}=i_a,\ 
\sum_{a=1}^k m_{a,b}=j_b
\right\}.
\eal
\item Let $\mathcal P=\{B_1,\dots,B_m\}$ be a set partition of $[n]$. The {\em Young subgroup indexed by $\mathcal{P}$} is defined as
\[
S_{\mathcal P} = \{\sigma\in S_n:\ \sigma(B_j)=B_j\ \text{for all }j=1,\dots,m\}.
\]
Thus $\sigma\in S_{\mathcal P}$ if and only if $\sigma$ permutes each block $B_j$ amongst itself.
\item Let $\mathcal P=\{P_i\}_{i\in I}$ and $\mathcal Q=\{Q_j\}_{j\in J}$ be set partitions of $[n]$. The {\em meet} of $\mathcal{P}$ and $\mathcal{Q}$ is
\bal
\mathcal P\wedge\mathcal Q =\{\,P_i\cap Q_j:\ i\in I,\ j\in J,\ P_i\cap Q_j\neq\varnothing\}.
\eal
\end{itemize}
We adopt the convention that $S_0 = \{e\}$ and $\ell(e) = 0$.

For a row vector $v=(v_1,\ldots,v_k) \in \R^k$, let
\[
 \|v\|_1:=\sum_{a=1}^k|v_a|,
 \qquad
 \|v\|_2^2:=\sum_{a=1}^k v_a^2.
\]
Let $e_a$ denote the $a$th standard row vector in $\mathbb R^k$. 

\subsection{Markov Chains and Mixing Times}

Let $X$ be a finite set. A Markov chain is specified by a {\em transition matrix}, or {\em transition kernel}, $K(x,y) \geq 0$ with $\sum_y K(x,y) = 1$, so that $K(x,y)$ is the probability of moving from $x$ to $y$ in one step. Under mild conditions, there exists a unique {\em stationary distribution} $\pi(x) \geq 0$ with $\sum_x \pi(x) = 1$, such that $\sum_x \pi(x) K(x,y) = \pi(y)$.

Let $K^t(x,y)$ be the probability of moving from $x$ to $y$ in $t$ steps, where $t$ is a nonnegative integer. The {\em fundamental theorem of Markov chains} states that if $K$ is irreducible and aperiodic, then $K_x^t(y) := K^t(x,y) \to \pi(y)$ as $t \to \infty$. The distance to stationarity can be measured by the {\em total variation distance}
\bal
\|K_x^t - \pi\|_{\TV} = \max_{A \subseteq X} |K^t(x,A) - \pi(A)| = \frac{1}{2} \sum_y |K^t(x,y) - \pi(y)|. 
\eal

Unless otherwise stated, all mixing-time asymptotics are for fixed $0<\epsilon<1$, and the eigenvalues of a finite reversible transition kernel $K$ are listed in decreasing order and counted with algebraic multiplicity, so $1=\lambda_1(K)\ge\lambda_2(K)\ge\cdots$. Its relaxation time is $t_{\rel}(K):=(1-\lambda_2(K))^{-1}$.

A partition $\mathcal P$ of $X$ is {\em strongly lumpable} for $K$ if, for every $B,C\in\mathcal P$, the sum $\sum_{y\in C}K(x,y)$ is independent of the choice of $x\in B$.

\subsection{Auxiliary Results}

In this section, we collect some technical lemmas and auxiliary results that we use in the paper.

\begin{proposition}\label{prop:young-product}
Let $\mathcal P=\{B_1,\dots,B_m\}$ be a set partition of $[n]$. Then
\[
S_{\mathcal P}\cong \prod_{j=1}^m S_{B_j},
\qquad
\sigma\longmapsto (\sigma|_{B_1},\dots,\sigma|_{B_m}),
\]
and hence $|S_{\mathcal P}|=\prod_{j=1}^m |B_j|!$.
\end{proposition}

\begin{proof}
Define
\[
\varphi:S_{\mathcal P}\to \prod_{j=1}^m S_{B_j},
\qquad
\varphi(\sigma):=(\sigma|_{B_1},\dots,\sigma|_{B_m}).
\]
This is well-defined since $\sigma(B_j)=B_j$ implies $\sigma|_{B_j}\in S_{B_j}$. Define
\[
\psi:\prod_{j=1}^m S_{B_j}\to S_{\mathcal P}
\]
by
\[
\psi(\pi_1,\dots,\pi_m)(r):=\pi_j(r)
\quad\text{for the unique }j\text{ with }r\in B_j.
\]
Then $\psi$ is well-defined because $\{B_1,\dots,B_m\}$ is a partition of $[n]$, and $\psi(\pi_1,\dots,\pi_m)(B_j)=B_j$ for every $j$, so $\psi(\pi_1,\dots,\pi_m)\in S_{\mathcal P}$. Also
\[
\varphi(\psi(\pi_1,\dots,\pi_m))=(\pi_1,\dots,\pi_m),
\qquad
\psi(\varphi(\sigma))=\sigma.
\]
Thus $\varphi$ is a bijection. For $\sigma,\tau\in S_{\mathcal P}$ and each $j$,
\[
(\sigma\tau)|_{B_j}=\sigma|_{B_j}\,\tau|_{B_j},
\]
so $\varphi(\sigma\tau)=\varphi(\sigma)\varphi(\tau)$; hence $\varphi$ is an isomorphism. Therefore
\bal
|S_{\mathcal P}| &= \prod_{j=1}^m |S_{B_j}|=\prod_{j=1}^m |B_j|!. \qedhere
\eal
\end{proof}

\begin{lemma}\label{lemma:young-meet}
Let $\mathcal P,\mathcal Q$ be set partitions of $[n]$. Then
\[
S_{\mathcal P}\cap S_{\mathcal Q}=S_{\mathcal P\wedge\mathcal Q}.
\]
\end{lemma}

\begin{proof}
If $\sigma\in S_{\mathcal P}\cap S_{\mathcal Q}$ and $D=B\cap C\neq\varnothing$ with $B\in\mathcal P$, $C\in\mathcal Q$, then
\[
\sigma(D)=\sigma(B\cap C)=\sigma(B)\cap\sigma(C)=B\cap C=D,
\]
so $\sigma\in S_{\mathcal P\wedge\mathcal Q}$.

Conversely, if $\sigma\in S_{\mathcal P\wedge\mathcal Q}$ and $B\in\mathcal P$, then
\[
B=\bigsqcup_{\substack{C\in\mathcal Q\\ B\cap C\neq\varnothing}} (B\cap C).
\]
Since each nonempty $B\cap C$ is a block of $\mathcal P\wedge\mathcal Q$,
\[
\sigma(B)=\bigsqcup_{\substack{C\in\mathcal Q\\ B\cap C\neq\varnothing}} \sigma(B\cap C)
=\bigsqcup_{\substack{C\in\mathcal Q\\ B\cap C\neq\varnothing}} (B\cap C)=B.
\]
Thus $\sigma\in S_{\mathcal P}$; by symmetry, $\sigma\in S_{\mathcal Q}$. Hence $\sigma\in S_{\mathcal P}\cap S_{\mathcal Q}$.
\end{proof}

For each $1\le i\le n$, let $E_i$ be a finite set and let $\mu_i$ and $\nu_i$ be probability measures on $E_i$. Let $\bigotimes_{i=1}^n\mu_i$ denote the product measure of $\mu_1, \ldots, \mu_n$. If $\mu,\nu$ are probability measures on a finite set $E$ and $\eta$ is a probability measure on a finite set $F$, then
\begin{equation}\label{eq:TV-common-factor}
 \|\mu\otimes\eta-\nu\otimes\eta\|_{\TV}
 =
 \frac12\sum_{x\in E}\sum_{z\in F}
 |\mu(x)-\nu(x)|\eta(z)
 =
 \|\mu-\nu\|_{\TV}.
\end{equation}
Set
\[
 P_j:=
 \left(\bigotimes_{r=1}^{j}\nu_r\right)
 \otimes
 \left(\bigotimes_{r=j+1}^{n}\mu_r\right),
 \qquad 0\le j\le n.
\]
Then \eqref{eq:TV-common-factor} and the triangle inequality give
\begin{equation}\label{eq:TV-product}
 \left\|
 \bigotimes_{r=1}^n\mu_r-
 \bigotimes_{r=1}^n\nu_r
 \right\|_{\TV}
 \le
 \sum_{j=1}^n\|P_{j-1}-P_j\|_{\TV}
 =
\sum_{j=1}^n\|\mu_j-\nu_j\|_{\TV}.
\end{equation}

Let $\gamma_1,\ldots,\gamma_k>0$ and let $\gamma_0:=\sum_{b=1}^k\gamma_b$. If 
\[
 (Z_1,\ldots,Z_k)\sim\Dirichlet(\gamma_1,\ldots,\gamma_k),
\]
then for nonnegative integers $r_1,\ldots,r_k$ with $r_0 = \sum_b r_b$,
\begin{equation}\label{eq:Dirichlet-moments}
 \mathbb E \left[\prod_{b=1}^k Z_b^{r_b} \right]
 =\frac{\prod_{b=1}^k(\gamma_b)_{(r_b)}}{(\gamma_0)_{(r_0)}}.
\end{equation}

\section{The Twisted Bose--Einstein Markov Chain}\label{SectionTransitionKernels}

This section derives the labeled and lumped twisted Bose--Einstein kernels and their stationary laws and identifies the labeled chain with Crane's Dirichlet cut-and-paste chain. Setting $\theta=1$ recovers the ordinary Bose--Einstein chain.

\subsection{The Labeled Twisted Bose--Einstein Markov Chain}

Let $K$ be the transition kernel of the twisted Bose--Einstein Markov chain. 
The next lemma supplies the fixed-set, stabilizer, and orbit counts needed to evaluate $K$ and $\pi$. Its parking-function assertions will also be used in the application developed below. 

\begin{lemma} \label{LemmaKeyFacts}
Let $n,k\ge1$ and $x,y \in [k]^n$.
\begin{enumerate}[label=(\roman*)]
\item For all $\sigma\in S_n$,
\bal
|([k]^n)_\sigma|=k^{\ell(\sigma)}.
\eal
When $k=n+1$, the restriction to parking functions satisfies
\bal
|(\PF_n)_\sigma|=(n+1)^{\ell(\sigma)-1}.
\eal
\item We have
\bal
G_x \cong S_{i_1(x)}\times\cdots\times S_{i_k(x)} \quad \text{and} \quad |G_x|=\prod_{a=1}^{k} i_a(x)!.
\eal
\item We have that
\bal
|[k]^n/S_n|=|\Comp_{n,k}|=\binom{n+k-1}{k-1}.
\eal
When $k=n+1$, the parking-function restriction satisfies
\bal
|\PF_n/S_n|=|\IPF_n|=C_n.
\eal
\item We have
\bal
G_x\cap G_y\cong\prod_{a,b=1}^k S_{m_{a,b}(x,y)} \quad \text{and} \quad
|G_x\cap G_y|=\prod_{a,b=1}^k m_{a,b}(x,y)!
\eal
\item For \(g\in G_x\cap G_y\), each nonempty \(B_{a,b}(x,y)\) is \(g\)-invariant, so
\[
g|_{B_{a,b}(x,y)}:B_{a,b}(x,y)\to B_{a,b}(x,y)
\]
is a permutation, and
\[
\ell(g)
=
\sum_{\substack{1\le a,b\le k\\ B_{a,b}(x,y)\neq\emptyset}}
\ell\!\left(g|_{B_{a,b}(x,y)}\right).
\]
\end{enumerate}
\end{lemma}

\begin{proof}
(i) A word in $[k]^n$ fixed by $\sigma$ is constant on each cycle of $\sigma$, and each cycle may be assigned any of the $k$ labels. Hence
\[
 |([k]^n)_\sigma|=k^{\ell(\sigma)}.
\]
The parking-function assertion is the parking-function permutation character value; see \cite[Proposition~2.2(a), Equation~(2)]{Sta97}.

(ii) For \(\sigma\in S_n\),
\[
\sigma\in G_x
\iff
\sigma x=x
\iff
x_{\sigma^{-1}(r)}=x_r \quad (r\in[n]).
\]
Thus \(\sigma\in G_x\) iff \(\sigma\) preserves each level set \(I_a(x)\). Hence
\[
G_x=\{\sigma\in S_n:\sigma(B)=B\ \text{for every }B\in\mathcal I(x)\}
=S_{\mathcal I(x)}.
\]
By Proposition~\ref{prop:young-product},
\[
G_x\cong\prod_{a=1}^k S_{I_a(x)}
\cong S_{i_1(x)}\times\cdots\times S_{i_k(x)},
\]
and therefore
\[
|G_x|=\prod_{a=1}^k i_a(x)!.
\]

(iii) If \(y=\sigma x\), then \(\sigma\) only reindexes coordinates, so
\[
\hist(y)=\hist(x).
\]
Conversely, suppose \(\hist(x)=\hist(y)\). Then
\[
|I_a(x)|=|I_a(y)| \qquad (1\le a\le k).
\]
Choose bijections \(\phi_a:I_a(x)\to I_a(y)\), and define \(\sigma\in S_n\) by
\[
\sigma|_{I_a(x)}=\phi_a.
\]
Then, for \(r\in I_a(y)\), we have \(\sigma^{-1}(r)\in I_a(x)\), so
\[
(\sigma x)_r=x_{\sigma^{-1}(r)}=a=y_r.
\]
Hence \(\sigma x=y\). Thus orbits are indexed by histograms. For each orbit, the weakly increasing rearrangement
\[
\bar x=
1^{i_1(x)}2^{i_2(x)}\cdots k^{i_k(x)}
\]
is its unique weakly increasing representative in $[k]^n$. Hence the orbits are indexed by $\Comp_{n,k}$. Therefore
\[
|[k]^n/S_n|=|\Comp_{n,k}|=\binom{n+k-1}{k-1}.
\]
For the parking-function restriction, the weakly increasing representative lies in $\IPF_n$, so the same argument shows that the orbits are indexed by $\IPF_n$. Therefore
\[
|\PF_n/S_n|=|\IPF_n|=C_n.
\]

(iv) By (ii), applied to \(x\) and \(y\),
\[
G_x=S_{\mathcal I(x)},
\qquad
G_y=S_{\mathcal I(y)}.
\]
By Lemma \ref{lemma:young-meet},
\[
G_x\cap G_y
=
S_{\mathcal I(x)}\cap S_{\mathcal I(y)}
=
S_{\mathcal I(x)\wedge\mathcal I(y)}.
\]
Since
\[
\mathcal I(x)\wedge\mathcal I(y)
=
\{B_{a,b}(x,y):B_{a,b}(x,y)\neq\emptyset\},
\]
Proposition~\ref{prop:young-product} gives
\[
G_x\cap G_y
\cong
\prod_{\substack{1\le a,b\le k\\B_{a,b}(x,y)\neq\emptyset}}
S_{B_{a,b}(x,y)}
\cong
\prod_{a,b=1}^k S_{m_{a,b}(x,y)}.
\]
Thus
\[
|G_x\cap G_y|
=
\prod_{a,b=1}^k m_{a,b}(x,y)!.
\]

(v) Let \(g\in G_x\cap G_y\), and let \(B_{a,b}(x,y)\neq\emptyset\). By (ii), applied to \(x\) and \(y\),
\[
g(I_a(x))=I_a(x),
\qquad
g(I_b(y))=I_b(y).
\]
Hence
\[
g(B_{a,b}(x,y))
=
g(I_a(x)\cap I_b(y))
=
g(I_a(x))\cap g(I_b(y))
=
I_a(x)\cap I_b(y)
=
B_{a,b}(x,y),
\]
where the second equality uses that \(g\) is bijective. Thus \(g|_{B_{a,b}(x,y)}\) is a permutation of \(B_{a,b}(x,y)\).

Since the nonempty \(B_{a,b}(x,y)\) form a partition of \([n]\), every cycle of \(g\) lies in exactly one such block. Therefore
\[
\ell(g) = \sum_{\substack{1\le a,b\le k\\B_{a,b}(x,y)\neq\emptyset}}
\ell\!\left(g|_{B_{a,b}(x,y)}\right). \qedhere
\]
\end{proof}

The formulas for the transition kernel and stationary distribution of the general twisted Burnside process now give the following explicit description of the twisted Bose--Einstein Markov chain.

\begin{proposition}
\label{prop:BE-kernel}
Let $K$ and $\pi$ be the transition kernel and the stationary distribution of the twisted Bose--Einstein Markov chain on $[k]^n$. Then for all $x,y \in [k]^n$,
\begin{equation}\label{eq:BE-kernel}
 K(x,y)
 =
 \prod_{a=1}^k
 \frac{\displaystyle\prod_{b=1}^k
 \left(\frac{\theta}{k}\right)_{(m_{a,b}(x,y))}}
 {(\theta)_{(i_a(x))}}
 \end{equation}
and
\begin{equation} \label{eq:BE-pi}
 \pi(x)
 =
 \frac{\displaystyle\prod_{a=1}^k(\theta)_{(i_a(x))}}
 {(k\theta)_{(n)}}.
\end{equation}
\end{proposition}

\begin{proof}
Lemma~\ref{LemmaKeyFacts}\textup{(i)} gives
\[
 V(\sigma)=|([k]^n)_\sigma|=k^{\ell(\sigma)}.
\]
Substituting $v\equiv1$, $w(\sigma)=\theta^{\ell(\sigma)}$, and this value of $V(\sigma)$ into \eqref{eq:general-twisted-Burnside-kernel} gives 
\begin{equation}\label{eq:twisted-group-sum}
 K(x,y)
 =\frac1{W(x)}
 \sum_{\sigma\in G_x\cap G_y}
 \left(\frac\theta k\right)^{\ell(\sigma)}.
\end{equation}
Equation~\eqref{eq:general-twisted-Burnside-pi} also gives
\[
 \pi(x)=\frac{W(x)}{Z},
 \qquad
 Z=\sum_{z\in[k]^n}W(z).
\]
By Lemma~\ref{LemmaKeyFacts}\textup{(ii)} and \eqref{eq:stirl1},
\[
 W(x)
 =\prod_{a=1}^k
 \sum_{\tau\in S_{i_a(x)}}\theta^{\ell(\tau)}
 =\prod_{a=1}^k(\theta)_{(i_a(x))}.
\]
Lemma~\ref{LemmaKeyFacts}\textup{(iv)--(v)} similarly gives
\begin{align*}
 \sum_{\sigma\in G_x\cap G_y}
 \left(\frac\theta k\right)^{\ell(\sigma)}
 &=\prod_{a,b=1}^k
 \sum_{\tau\in S_{m_{a,b}(x,y)}}
 \left(\frac\theta k\right)^{\ell(\tau)} =\prod_{a,b=1}^k
 \left(\frac\theta k\right)_{(m_{a,b}(x,y))}.
\end{align*}
This proves the kernel formula. Moreover,
\begin{align*}
 \sum_{x\in[k]^n}W(x)
 &=\sum_{x\in[k]^n}\sum_{\sigma\in G_x}
 \theta^{\ell(\sigma)}=\sum_{\sigma\in S_n}\sum_{x\in([k]^n)_\sigma}
 \theta^{\ell(\sigma)}\\
 &=\sum_{\sigma\in S_n}
 \theta^{\ell(\sigma)}|([k]^n)_\sigma|
 =\sum_{\sigma\in S_n}(k\theta)^{\ell(\sigma)}
 =(k\theta)_{(n)},
\end{align*}
where the last equality is \eqref{eq:stirl1}. Together with the formula for $W(x)$, this proves the stationary formula.
\end{proof}

Setting $\theta=1$, the twisted Bose--Einstein Markov chain recovers the ordinary Bose--Einstein Markov chain on $[k]^n$. Denote its transition kernel and stationary distribution by $K_{\BE}$ and $\pi_{\BE}$, respectively.

\begin{corollary}
\label{cor:ordinary-BE-kernel}
Let $K_{\BE}$ and $\pi_{\BE}$ be the transition kernel and the stationary distribution of the Bose--Einstein Markov chain on $[k]^n$. Then for all $x,y \in [k]^n$,
\begin{equation}\label{eq:ordinary-BE-kernel}
 K_{\BE}(x,y)
 =\frac{1}{\prod_{a=1}^k i_a(x)!}
 \prod_{a,b=1}^k
 \left(\frac1k\right)_{(m_{a,b}(x,y))}
\end{equation}
and
\begin{equation}\label{eq:ordinary-BE-pi}
 \pi_{\BE}(x)
 =\frac{\prod_{a=1}^k i_a(x)!}
 {n!\binom{n+k-1}{k-1}}
 =\frac{\prod_{a=1}^k i_a(x)!}{(k)_{(n)}}.
\end{equation}
\end{corollary}

\begin{proof}
Set $\theta=1$ in Proposition~\ref{prop:BE-kernel} and use
\[
 (1)_{(i_a(x))}=i_a(x)!,
 \qquad
 (k)_{(n)}=n!\binom{n+k-1}{k-1}. \qedhere
\]
\end{proof}

\begin{remark}[Equivalence with Crane's symmetric Dirichlet cut-and-paste process]
\label{rem:cut-and-paste}
The symmetric Dirichlet cut-and-paste process in Crane~\cite[Example~1.7, Equation~(1.8)]{Crane14} has independent row laws
\[
 \Dirichlet\left(\frac\theta k,\ldots,\frac\theta k\right), 
\]
and has exactly the transition kernel \eqref{eq:BE-kernel} and stationary distribution \eqref{eq:BE-pi}. Therefore, Crane's cut-and-paste process and the twisted Bose--Einstein Markov chain have identical transition kernels and stationary distributions, not merely identical orbit lumpings. We use this equivalence later to obtain the Dirichlet random-matrix representation of $K$.
\end{remark}

\subsection{The Lumped Twisted Bose--Einstein Markov Chain}
\label{SectionLumpedBE}

By Lemma~\ref{LemmaKeyFacts}\textup{(iii)}, the histogram map $\hist:[k]^n\to\Comp_{n,k}$ identifies the $S_n$-orbit space with $\Comp_{n,k}$. Thus, for $\bm{i}\in\Comp_{n,k}$,
\[
 O_{\bm{i}}=\hist^{-1}(\bm{i})
 =\{x\in[k]^n:\hist(x)=\bm{i}\}.
\]
The coordinate-permutation invariance of \eqref{eq:BE-kernel} makes this partition strongly lumpable. Thus, for any $x\in O_{\bm{i}}$, define
\[
 \overline K(\bm{i},\bm{j})
 :=\sum_{z\in O_{\bm{j}}}K(x,z),
 \qquad
 \overline\pi(\bm{i})
 :=\sum_{z\in O_{\bm{i}}}\pi(z).
\]
For every function $h:\Comp_{n,k}\to\mathbb R$, strong lumpability gives
\begin{equation}\label{eq:histogram-intertwining}
 K(h\circ\hist)=(\overline Kh)\circ\hist,
 \qquad
 \pi(h\circ\hist)=\overline\pi(h).
\end{equation}
Moreover, pushing both measures through $\hist$ gives
\begin{equation}\label{eq:histogram-TV-contraction}
 \left\|\overline K^t(\hist(x),\cdot)-\overline\pi\right\|_{\TV}
 \le
 \left\|K^t(x,\cdot)-\pi\right\|_{\TV}.
\end{equation}
Set
\[
 \overline d(t):=\max_{\bm{i}\in\Comp_{n,k}}
 \left\|\overline K^t(\bm{i},\cdot)-\overline\pi\right\|_{\TV},
 \qquad
 \overline t_{\mix}(\varepsilon):=\min\{t\ge0:\overline d(t)\le\varepsilon\}.
\]

\begin{proposition}[Lumped twisted Bose--Einstein kernel]
\label{prop:lumped-BE-kernel}
Let $\overline K$ and $\overline\pi$ be the transition kernel and stationary distribution of the lumped twisted Bose--Einstein Markov chain on $\Comp_{n,k}$. Then for all $\bm{i},\bm{j}\in\Comp_{n,k}$,
\begin{equation}\label{eq:lumped-BE-kernel}
 \overline K(\bm{i},\bm{j})
 =
 \sum_{M=(m_{a,b})\in\mc(\bm{i},\bm{j})}
 \prod_{a=1}^k\frac{i_a!}{(\theta)_{(i_a)}}
 \prod_{a,b=1}^k
 \frac{(\theta/k)_{(m_{a,b})}}{m_{a,b}!}
\end{equation}
and
\begin{equation}\label{eq:lumped-BE-pi}
 \overline\pi(\bm{i})
 =
 \frac{n!}{(k\theta)_{(n)}}
 \prod_{a=1}^k\frac{(\theta)_{(i_a)}}{i_a!}.
\end{equation}
In particular, $\overline\pi$ is the Dirichlet--multinomial distribution with parameters $n$ and $(\theta,\ldots,\theta)$.
\end{proposition}

\begin{proof}
Fix $x\in O_{\bm{i}}$. By Proposition~\ref{prop:BE-kernel},
\[
 \overline K(\bm{i},\bm{j})
 =
 \sum_{\substack{z\in[k]^n\\\hist(z)=\bm{j}}}
 \prod_{a=1}^k
 \frac{\prod_{b=1}^k
 (\theta/k)_{(m_{a,b}(x,z))}}
 {(\theta)_{(i_a)}}.
\]
For every such $z$, the contingency table $\dens(x,z)$ belongs to $\mc(\bm{i},\bm{j})$, and
\[
 O_{\bm j}
 =\bigsqcup_{M\in\mc(\bm i,\bm j)}
 \{z\in[k]^n:\dens(x,z)=M\}.
\]
Fix $M=(m_{a,b})\in\mc(\bm{i},\bm{j})$. The words $z$ satisfying $\dens(x,z)=M$ are obtained, independently for each $a$, by assigning the labels $1,\ldots,k$ to the $i_a$ coordinates in $I_a(x)$ with multiplicities $m_{a,1},\ldots,m_{a,k}$. Hence
\[
 \bigl|\{z\in[k]^n:\dens(x,z)=M\}\bigr|
 =
 \prod_{a=1}^k
 \binom{i_a}{m_{a,1},\ldots,m_{a,k}}
 =
 \frac{\prod_{a=1}^k i_a!}
 {\prod_{a,b=1}^k m_{a,b}!}.
\]
Therefore, grouping the target words according to $M=\dens(x,z)$ gives
\begin{align*}
 \overline K(\bm{i},\bm{j})
 &=
 \sum_{M=(m_{a,b})\in\mc(\bm{i},\bm{j})}
 \frac{\prod_{a=1}^k i_a!}
 {\prod_{a,b=1}^k m_{a,b}!}
 \prod_{a=1}^k
 \frac{\prod_{b=1}^k(\theta/k)_{(m_{a,b})}}
 {(\theta)_{(i_a)}}\\
 &=
 \sum_{M=(m_{a,b})\in\mc(\bm{i},\bm{j})}
 \prod_{a=1}^k\frac{i_a!}{(\theta)_{(i_a)}}
 \prod_{a,b=1}^k
 \frac{(\theta/k)_{(m_{a,b})}}{m_{a,b}!},
\end{align*}
which proves \eqref{eq:lumped-BE-kernel}.

For the stationary distribution, \eqref{eq:BE-pi} is constant on $O_{\bm{i}}$, and
\[
 |O_{\bm{i}}|=\frac{n!}{\prod_{a=1}^k i_a!}.
\]
Therefore
\[
 \overline\pi(\bm{i})
 =|O_{\bm{i}}|\pi(x)
 =
 \frac{n!}{(k\theta)_{(n)}}
 \prod_{a=1}^k\frac{(\theta)_{(i_a)}}{i_a!}. \qedhere
\]
\end{proof}

Setting $\theta=1$ yields the corresponding transition kernel, $\overline K_{\BE}$, and stationary distribution, $\overline\pi_{\BE}$, of the lumped Bose--Einstein Markov chain.  

\begin{corollary}
\label{cor:ordinary-lumped-BE-kernel}
Let $\overline K_{\BE}$ and $\overline\pi_{\BE}$ be the transition kernel and stationary distribution of the lumped Bose--Einstein Markov chain on $\Comp_{n,k}$. Then for all $\bm{i},\bm{j}\in\Comp_{n,k}$,
\begin{equation}\label{eq:ordinary-lumped-BE-kernel}
 \overline K_{\BE}(\bm{i},\bm{j})
 =
 \sum_{M=(m_{a,b})\in\mc(\bm{i},\bm{j})}
 \prod_{a,b=1}^k
 \frac{(1/k)_{(m_{a,b})}}{m_{a,b}!},
\end{equation}
and
\begin{equation}\label{eq:ordinary-lumped-BE-pi}
 \overline\pi_{\BE}(\bm{i})
 =
 \frac1{\binom{n+k-1}{k-1}}.
\end{equation}
\end{corollary}

\begin{proof}
Set $\theta=1$ in Proposition~\ref{prop:lumped-BE-kernel} and use $(1)_{(i_a)}=i_a!$ and $(k)_{(n)}=n!\binom{n+k-1}{k-1}$.
\end{proof}

\begin{remark}[Recovery of Diaconis's formulas]
Corollary~\ref{cor:ordinary-BE-kernel} gives the explicit labeled transition kernel for the update in \cite[Equation~(2.1)]{Dia05}, while Corollary~\ref{cor:ordinary-lumped-BE-kernel} gives the orbit chain studied there. For $k=2$, \eqref{eq:ordinary-lumped-BE-kernel} gives Equations~(3.1)--(3.2) of \cite{Dia05}; reversibility and label exchange give the symmetries in Equation~(3.3). For general $k$, starting from the constant histogram $(n,0,\ldots,0)$ leaves a unique contingency table and gives
\[
 \overline K_{\BE}\bigl((n,0,\ldots,0),\bm{j}\bigr)
 =
 \prod_{b=1}^k\frac{(1/k)_{(j_b)}}{j_b!},
\]
which is Equation~(3.6) of \cite{Dia05}. Equation~\eqref{eq:ordinary-lumped-BE-pi} is the uniform Bose--Einstein stationary law used there. Thus these transition and stationary formulas follow from the general formulas above.
\end{remark}

\section{Application to the Burnside Process on Parking Functions}
\label{SectionComparison}

In this section, we compute the transition kernels for the Burnside process on $\PF_n$ and the lumped Burnside process on $\IPF_n$. We obtain these formulas from the preceding labeled and lumped chains at $\theta=1$ and $k=n+1$, together with the rescaling relation below. Pollak's cyclic-shift bijection then upgrades the comparison to exact tensor factorizations, which transfer the corresponding mixing and spectral results to the parking-function chain and its lumped process.

\subsection{Transition Kernel for the Burnside Process on \texorpdfstring{$\PF_n$}{PF\_n}}

Regard $\PF_n$ as a subset of $[n+1]^n$ via the inclusion map $[n]\hookrightarrow[n+1]$.
The following lemma gives a comparison of the transition kernels and stationary distributions between the Bose--Einstein Markov chain and the Burnside process on $\PF_n$. 

\begin{lemma}\label{lem:PF-vs-BE}
Let $K_{\PF}$ and $\pi_{\PF}$ be the transition kernel and stationary distribution of the Burnside process on $\PF_n$. Let $K_{\BE}$ and $\pi_{\BE}$ be the transition kernel and stationary distribution of the Bose--Einstein Markov chain on $[n+1]^n$. Then for all $x,y\in\PF_n$,
\[
 K_{\PF}(x,y)=(n+1)K_{\BE}(x,y)
 \quad \text{and} \quad 
 \pi_{\PF}(x)=(n+1)\pi_{\BE}(x).
\]
\end{lemma}

\begin{proof}
In both state spaces, $S_n$ acts by permuting coordinates, so the stabilizer $G_x$ is the same. By \eqref{BPTransEqn},
\begin{align*}
 K_{\PF}(x,y)
 &=\frac1{|G_x|}\sum_{\sigma\in G_x\cap G_y}
 \frac1{|(\PF_n)_\sigma|},\\
 K_{\BE}(x,y)
 &=\frac1{|G_x|}\sum_{\sigma\in G_x\cap G_y}
 \frac1{|([n+1]^n)_\sigma|}.
\end{align*}
Lemma~\ref{LemmaKeyFacts}\textup{(i)} gives
\[
 |([n+1]^n)_\sigma|=(n+1)^{\ell(\sigma)},
 \qquad
 |(\PF_n)_\sigma|=(n+1)^{\ell(\sigma)-1},
\]
and hence $K_{\PF}(x,y)=(n+1)K_{\BE}(x,y)$. For the stationary laws, the $S_n$-orbit of $x\in\PF_n$ is unchanged when $x$ is viewed in $[n+1]^n$, since $\PF_n$ is invariant under coordinate permutations. By Lemma~\ref{LemmaKeyFacts}\textup{(iii)}, the two actions have $C_n$ and
\[
 |\Comp_{n,n+1}|=\binom{2n}{n}=(n+1)C_n
\]
orbits, respectively, where $C_n$ is the $n$th Catalan number. Therefore \eqref{BPStationaryEqn} gives
\[
 \pi_{\PF}(x)=(n+1)\pi_{\BE}(x). \qedhere
\]
\end{proof}

The Bose--Einstein formulas immediately give the following corollary.

\begin{corollary}\label{BPTransition}
Let $K_{\PF}$ and $\pi_{\PF}$ be the transition kernel and stationary distribution of the Burnside process on $\PF_n$. Then for all $x,y\in \PF_n$,
\bal
K_{\PF}(x,y) = \frac{n+1}{\prod_{a=1}^n i_a(x)!} \prod_{a,b = 1}^n \left(\frac{1}{n+1}\right)_{(m_{a,b}(x,y))}
\eal
and
\bal
\pi_{\PF}(x) = \frac{\prod_{a=1}^n i_a(x)!}{n! C_n},
\eal
where $C_n$ is the $n$th Catalan number.
\end{corollary}

\begin{proof}
Apply Lemma~\ref{lem:PF-vs-BE} and Corollary~\ref{cor:ordinary-BE-kernel} with $k=n+1$; the unused label $n+1$ contributes only unit factors.
\end{proof}

Thus the parking-function kernel is the rescaled restriction to $\PF_n$ of the twisted Bose--Einstein kernel at $\theta=1$, which is the ordinary Bose--Einstein kernel. The transfer theorem below upgrades this relation to an exact tensor factorization.

\subsection{Transition Kernel for the Lumped Burnside Process on \texorpdfstring{$\IPF_n$}{IPF\_n}}

Recall that the $S_n$-orbits of $\PF_n$ are indexed by $\IPF_n$. The general lumped Bose--Einstein formula gives the following corollary.

\begin{corollary}\label{BPLumpedTrans}
Let $\overline K_{\PF}$ and $\overline\pi_{\PF}$ be the transition kernel and uniform stationary distribution of the lumped Burnside process on $\IPF_n$. 
For $u,v\in\IPF_n$, write $\bm{i}:=\hist(u)$ and $\bm{j}:=\hist(v)$. The lumped Burnside process on $\IPF_n$ has transition kernel
\begin{align}
 \overline K_{\PF}(u,v)
 &=(n+1)\overline K_{\BE}\bigl((\bm{i},0),(\bm{j},0)\bigr)\notag\\
 &=
 (n+1)
 \sum_{M=(m_{a,b})\in\mc(\bm{i},\bm{j})}
 \prod_{a,b=1}^n
 \frac{(1/(n+1))_{(m_{a,b})}}{m_{a,b}!}.
 \label{eq:PF-lumped-kernel}
\end{align}
Its stationary distribution is
\[
 \overline\pi_{\PF}(u)
 =(n+1)\overline\pi_{\BE}\bigl((\bm{i},0)\bigr)
 =\frac1{C_n}.
\]
\end{corollary}

\begin{proof}
Coordinate permutations preserve parking functions. Hence the parking functions with histogram $\bm j$ form the full $S_n$-orbit of $v$, which, under $\PF_n\subset[n+1]^n$, is the Bose--Einstein orbit with histogram $(\bm j,0)$. Summing Lemma~\ref{lem:PF-vs-BE} over this orbit gives
\begin{align*}
 \overline K_{\PF}(u,v)
 &=
 \sum_{\substack{z\in\PF_n:\\ \hist(z)=\bm{j}}}K_{\PF}(u,z) =
 (n+1)
 \overline K_{\BE}\bigl((\bm{i},0),(\bm{j},0)\bigr).
\end{align*}
Every contingency table with row sums $(\bm i,0)$ and column sums $(\bm j,0)$ has zero final row and column, so \eqref{eq:ordinary-lumped-BE-kernel} reduces to \eqref{eq:PF-lumped-kernel}. Similarly, Lemma~\ref{lem:PF-vs-BE} and \eqref{eq:ordinary-lumped-BE-pi} give
\[
 \overline\pi_{\PF}(u)
 =
 (n+1)\overline\pi_{\BE}\bigl((\bm{i},0)\bigr)
 =
 \frac{n+1}{\binom{2n}{n}}
 =
 \frac1{C_n}. \qedhere
\]
\end{proof}

\subsection{Transfer by Global Cyclic Shifts}

Identify $[n+1]=\{1,\ldots,n+1\}$ with $H:=\mathbb Z_{n+1}$ via $a\mapsto a-1\pmod{n+1}$, and write
\[
 c\cdot(x_1,\ldots,x_n):=(x_1+c,\ldots,x_n+c)
\]
for the global shift action of $c\in H$ on $[n+1]^n$. If $c\cdot x=x$, then $x_1+c=x_1$ in $H$, hence $c=0$. Thus the action is free, so
\[
 |[x]_H|=n+1
 \quad\text{and}\quad
 |[n+1]^n/H|=(n+1)^{n-1}.
\]
By Pollak's circular argument \cite[Sec.~2, p.~3]{Sta97}, every $H$-orbit in $[n+1]^n$ contains exactly one parking function. Let
\[
 \Phi:[n+1]^n\longrightarrow\PF_n
\]
denote this unique representative map. Equivalently,
\[
 \Psi:\PF_n\times H\longrightarrow[n+1]^n,
 \qquad
 \Psi(u,c):=c\cdot u,
\]
is a bijection. Finally let $U_H$ denote the {\em uniform refresh} kernel on $H$, so that $U_H(c,d)=1/(n+1)$ for all $c,d\in H$.

The next lemma shows that global shifts are invisible to the Bose--Einstein transition and stationary laws, allowing us to pass to the parking-function quotient.

\begin{lemma}\label{lem:BE-shift}
For all $c,d\in H$ and $x,y\in[n+1]^n$,
\[
 G_{c\cdot x}=G_x,
 \qquad
 \pi_{\BE}(c\cdot x)=\pi_{\BE}(x),
 \qquad
 K_{\BE}(c\cdot x,d\cdot y)=K_{\BE}(x,y).
\]
\end{lemma}

\begin{proof}
Coordinate permutations commute with global shifts: $\sigma(c\cdot x)=c\cdot(\sigma x)$. Hence $G_{c\cdot x}=G_x$. With indices interpreted modulo $n+1$,
\[
 i_a(c\cdot x)=i_{a-c}(x),
 \qquad
 m_{a,b}(c\cdot x,d\cdot y)
 =m_{a-c,b-d}(x,y).
\]
Thus global shifts merely permute the factors in \eqref{eq:ordinary-BE-kernel} and \eqref{eq:ordinary-BE-pi}, proving the kernel and stationary-distribution identities.
\end{proof}

Define the $H$-orbit projection and pushforward stationary measure by
\[
 \overline K_{\BE,H}([x]_H,[y]_H)
 :=\sum_{u\in[y]_H}K_{\BE}(x,u),
 \qquad
 \overline\pi_{\BE,H}([x]_H)
 :=\sum_{u\in[x]_H}\pi_{\BE}(u).
\]
Let $d_{\IPF}(t)$ and $t_{\mix,\IPF}(\varepsilon)$ be the distance function and mixing time of $\overline K_{\PF}$.

\begin{theorem}\label{thm:BE-PF-transfer}
Let $x,y\in[n+1]^n$. Under the identification $[x]_H\mapsto\Phi(x)$,
\[
 \overline K_{\BE,H}([x]_H,[y]_H)
 =K_{\PF}(\Phi(x),\Phi(y)),
 \qquad
 \overline\pi_{\BE,H}([y]_H)=\pi_{\PF}(\Phi(y)).
\]
Moreover, under the bijection $\Psi$,
\begin{equation}\label{eq:tensor}
 K_{\BE}=K_{\PF}\otimes U_H,
 \qquad
 \pi_{\BE}=\pi_{\PF}\otimes\Unif(H).
\end{equation}
Consequently, with $\overline d$ specialized to $\theta=1$ and $k=n+1$, for all $t\ge1$,
\begin{align}
 \left\|(K_{\BE})^t(x,\cdot)-\pi_{\BE}\right\|_{\TV}
 &=\left\|(K_{\PF})^t(\Phi(x),\cdot)-\pi_{\PF}\right\|_{\TV},
 \label{eq:PF-BE-distance}\\
 \overline d(t)&=d_{\IPF}(t).
 \label{eq:lumped-PF-BE-distance}
\end{align}
The labeled kernels, and likewise the lumped kernels, have the same nonzero eigenvalues with multiplicity.
\end{theorem}

\begin{proof}
For $u,v\in\PF_n$ and $c,d\in H$, Lemmas~\ref{lem:BE-shift} and \ref{lem:PF-vs-BE} give
\[
 K_{\BE}(c\cdot u,d\cdot v)
 =K_{\BE}(u,v)
 =\frac1{n+1}K_{\PF}(u,v),
\]
and $\pi_{\BE}(c\cdot u)=\pi_{\PF}(u)/(n+1)$. This proves \eqref{eq:tensor}.

Let $u:=\Phi(x)$ and $v:=\Phi(y)$. Summing the two tensor identities over the $H$-coordinate gives
\[
 \overline K_{\BE,H}([x]_H,[y]_H)=K_{\PF}(u,v),
 \qquad
 \overline\pi_{\BE,H}([y]_H)=\pi_{\PF}(v).
\]
Since $U_H^t=U_H$ for $t\ge1$, \eqref{eq:tensor} gives
\[
 (K_{\BE})^t=(K_{\PF})^t\otimes U_H.
\]
The common-factor identity \eqref{eq:TV-common-factor} gives \eqref{eq:PF-BE-distance}. Since $U_H$ has eigenvalues $1,0,\ldots,0$, the tensor-product rule for eigenvalues proves the labeled spectral assertion. Finally, $\Psi$ is $S_n$-equivariant, so \eqref{eq:tensor} descends to the coordinate-permutation orbits; the same argument proves \eqref{eq:lumped-PF-BE-distance} and the lumped spectral assertion.
\end{proof}

\section{Mixing and Spectral Bounds}
\label{SectionMixingSpectral}

\subsection{Dirichlet Product Representation}

By Remark~\ref{rem:cut-and-paste}, the twisted Bose--Einstein Markov chain is equivalent to Crane's symmetric Dirichlet cut-and-paste process. Accordingly, let $S$ be a random row-stochastic $k\times k$ matrix with independent Dirichlet distributed rows,
\[
 S_{a,\bullet}\sim
 \Dirichlet\left(\frac\theta k,\ldots,\frac\theta k\right).
\]
For a deterministic row-stochastic $k\times k$ matrix $A$ and $x,y\in[k]^n$, define the {\em auxiliary kernel}
\[
 Q_A(x,y):=\prod_{r=1}^{n}A_{x_ry_r}.
\]
This kernel has the following probabilistic interpretation. Given $A$ and a current word $X=x$, generate the coordinates of the next word $Y$ independently according to
\[
 \mathbb P(Y_r=b\mid X=x,A)=A_{x_rb}.
\]
Consequently,
\begin{align*}
 \mathbb P(Y=y\mid X=x,A)
 =\prod_{r=1}^n\mathbb P(Y_r=y_r\mid X=x,A) =\prod_{r=1}^n A_{x_ry_r}
 =Q_A(x,y).
\end{align*}
Thus $Q_A$ is the transition kernel for the coordinatewise update governed by $A$.

\begin{lemma}\label{lem:Q-AB}
If $A$ and $B$ are deterministic row-stochastic $k\times k$ matrices, then
\[
 Q_AQ_B=Q_{AB}.
\]
\end{lemma}

\begin{proof}
By direct computation,
\begin{align*}
 (Q_AQ_B)(x,y)
 &=
 \sum_{z_1=1}^{k}\cdots\sum_{z_n=1}^{k}
 \prod_{r=1}^{n}A_{x_rz_r}B_{z_ry_r} =
 \prod_{r=1}^{n}
 \left(\sum_{c=1}^{k}A_{x_rc}B_{cy_r}\right)
 =Q_{AB}(x,y). \qedhere
\end{align*}
\end{proof}

The next proposition gives an equivalent construction of the chain. At each step, sample an independent stochastic matrix with the Dirichlet distributed rows as above, and conditional on that matrix, update the coordinates independently according to its coordinate-product kernel. Thus, one step averages $Q_S$ over $S$ and successive steps multiply the sampled matrices.

\begin{proposition}[Dirichlet product representation]
\label{prop:finite-dirichlet}
Let $K$ be the transition kernel of the twisted Bose--Einstein Markov chain on $[k]^n$. 
Let $S$ be a random row-stochastic $k\times k$ matrix with independent Dirichlet distributed rows, such that $S_{a,\bullet}\sim
 \Dirichlet\left(\frac\theta k,\ldots,\frac\theta k\right)$ for all $1 \leq a \leq k$. Let $S_1,S_2,\ldots$ be independent and identically distributed copies of $S$. Let
\[
 R_t:=S_1\cdots S_t \quad \text{and} 
 \quad R_0:=I_k.
\]
Then
\begin{equation}\label{eq:mixture}
 K=\mathbb E[Q_S],
\end{equation}
and
\begin{equation}\label{eq:matrix-product}
 K^t=\mathbb E[Q_{R_t}]
 \qquad(t\ge0).
\end{equation}
Equivalently, if $(X_t)_{t\ge0}$ is the twisted Bose--Einstein chain, then the chain can be realized so that
\begin{equation}\label{eq:conditional-product}
 \mathcal L(X_t\mid S_1,\ldots,S_t,X_0=x)
 =
 \bigotimes_{r=1}^{n}(e_{x_r}R_t).
\end{equation}
Here $\mathcal L(X)$ denotes the law of a random variable $X$.
\end{proposition}

\begin{proof}
For each row $a$ of $S$, the Dirichlet moment formula \eqref{eq:Dirichlet-moments} gives
\[
 \mathbb E \left[\prod_{b=1}^{k}S_{ab}^{m_{a,b}(x,y)} \right]
 =
 \frac{\displaystyle\prod_{b=1}^{k}
 (\theta/k)_{(m_{a,b}(x,y))}}
 {(\theta)_{(i_a(x))}}.
\]
Using independence of the rows and \eqref{eq:BE-kernel},
\begin{align*}
 \mathbb E [Q_S(x,y)]
 =
 \mathbb E\left[\prod_{a,b=1}^{k}
 S_{ab}^{m_{a,b}(x,y)}\right]
 =
 \prod_{a=1}^{k}
 \frac{\displaystyle\prod_{b=1}^{k}
 (\theta/k)_{(m_{a,b}(x,y))}}
 {(\theta)_{(i_a(x))}}
 =K(x,y).
\end{align*}
This proves \eqref{eq:mixture}. Lemma~\ref{lem:Q-AB} and independence of $S_1,\ldots,S_t$ imply \eqref{eq:matrix-product}. Finally,
\[
 Q_{R_t}(x,y)
 =\prod_{r=1}^{n}(R_t)_{x_ry_r}
 =\prod_{r=1}^{n}(e_{x_r}R_t)_{y_r}
 =\left(\bigotimes_{r=1}^{n}e_{x_r}R_t\right)(y),
\]
proving (\ref{eq:conditional-product}).
\end{proof}

\begin{remark}[Relation to the binary Bose--Einstein Markov chain]\label{rem:dlr}
For $\theta=1$ and $k=2$, the identity
\[
 \left(\frac12\right)_{(m)}=\frac{(2m)!}{4^m m!}
\]
shows that \eqref{eq:BE-kernel} recovers Diaconis--Lin--Ram's formula in Proposition~3.1 of \cite{DLR25a+}; the row-and-column relabeling symmetry recovers their Corollary~3.2.
Moreover, let $K_J$ denote the same twisted chain restricted to the coordinates in a nonempty set $J\subseteq[n]$. Then, for every $v\in[k]^J$ and $t\ge1$,
\[
 \sum_{y:\,y_J=v}K^t(x,y)=K_J^t(x_J,v).
\]
Indeed, summing \eqref{eq:matrix-product} over each omitted coordinate contributes $\sum_c(R_t)_{x_rc}=1$. At $\theta=1$, this recovers Proposition~3.3 of \cite{DLR25a+}; its one-coordinate algebraic form is Proposition~6.1 of \cite{DLR25b+}.
\end{remark}

For the remainder of this section, set
\[
\rho:=\frac{k-1}{k(\theta+1)}.
\]

The following technical lemma shows that $R_t$ contracts zero sum vectors in mean square. We will use it to prove the mixing time upper bound in the next subsection. 

\begin{lemma}[Mean-square contraction]\label{lem:contraction}
If $v\in\mathbb R^k$ satisfies $\sum_av_a=0$, then
\begin{equation}\label{eq:one-step-contraction}
 \mathbb E\|vS\|_2^2
 =\rho\|v\|_2^2.
\end{equation}
Consequently,
\begin{equation}\label{eq:contraction}
 \mathbb E\|vR_t\|_2^2
 =\rho^t\|v\|_2^2.
\end{equation}
\end{lemma}

\begin{proof}
The Dirichlet moments used in Proposition~\ref{prop:finite-dirichlet} give
\[
 \mathbb ES_{ab}=\frac1k,
 \qquad
 \mathbb ES_{ab}^2=\frac{\theta+k}{k^2(\theta+1)}.
\]
Since distinct rows are independent,
\begin{equation}\label{eq:row-inner-product}
 \mathbb E\langle S_{a,\bullet},S_{c,\bullet}\rangle
 =\sum_{b=1}^k\mathbb E[S_{ab}S_{cb}]
 =
 \frac1k+\rho\mathbf1_{\{a=c\}}.
\end{equation}
Therefore,
\begin{align*}
 \mathbb E\|vS\|_2^2
 =
 \sum_{a,c=1}^k v_av_c
 \mathbb E\langle S_{a,\bullet},S_{c,\bullet}\rangle =
 \frac1k\left(\sum_{a=1}^kv_a\right)^2
 +\rho\sum_{a=1}^kv_a^2
 =\rho\|v\|_2^2.
\end{align*}
This proves \eqref{eq:one-step-contraction}. Multiplication by a row-stochastic matrix preserves the coordinate sum, so $vR_{t-1}$ has coordinate sum zero. Since $S_t$ is independent of $S_1,\ldots,S_{t-1}$, \eqref{eq:one-step-contraction} gives
\begin{align*}
 \mathbb E\!\left[\|vR_t\|_2^2\mid S_1,\ldots,S_{t-1}\right]
 &=\mathbb E\!\left[\|(vR_{t-1})S_t\|_2^2\mid S_1,\ldots,S_{t-1}\right]\\
 &=\rho\|vR_{t-1}\|_2^2.
\end{align*}
Taking expectations and iterating from $R_0=I_k$ yields
\[
 \mathbb E\|vR_t\|_2^2=\rho^t\|v\|_2^2. \qedhere
\]
\end{proof}

\subsection{Collision Eigenfunctions and the Second Largest Eigenvalue}

We call $x_r=x_s$, meaning that coordinates $r$ and $s$ have the same label, a collision.

For distinct $r,s\in[n]$, set
\[
 A_{r,s}:=\{x\in[k]^n:x_r=x_s\},
 \qquad
 H_{r,s}:=\mathbf1_{A_{r,s}},
 \qquad
 c:=\frac{\theta+1}{k\theta+1},
 \qquad
 F_{r,s}:=H_{r,s}-c.
\]
For $n\ge2$ and $\bm{i}\in\Comp_{n,k}$, also set
\[
 J(\bm{i}):=\binom{n}{2}^{-1}\sum_{a=1}^k\binom{i_a}{2}.
\]

\begin{proposition}
\label{prop:collision-eigenfunction}
Assume $n,k\ge2$. 
Let $K$ and $\pi$ be the transition kernel and stationary distribution of the twisted Bose--Einstein Markov chain on $[k]^n$. 
Then
\begin{equation}\label{eq:collision-eigenfunction}
 K F_{r,s}=\rho F_{r,s},
\end{equation}
and, for every $x\in[k]^n$ and $t\ge0$,
\begin{equation}\label{eq:exact-two-coordinate-correlation}
 \mathbb P_x(X_t(r)=X_t(s))
 =c+\rho^tF_{r,s}(x).
\end{equation}
Moreover,
\[
 \pi(x_r=x_s)=c,
\]
and
\begin{equation}\label{eq:pair-state-lower}
 \left\|K^t(x,\cdot)-\pi\right\|_{\TV}
 \ge
 \rho^t
 \left|H_{r,s}(x)-c\right|.
\end{equation}
If $x_{\mathrm{const}}$ is any constant word, then
\begin{equation}\label{eq:constant-word-lower}
 \left\|K^t(x_{\mathrm{const}},\cdot)-\pi\right\|_{\TV}
 \ge
 \frac{\theta(k-1)}{k\theta+1}\rho^t.
\end{equation}
Consequently,
\begin{equation}\label{eq:collision-lower}
 d(t)
 \ge
 \frac{\max\{\theta+1,\theta(k-1)\}}
 {k\theta+1}\rho^t.
\end{equation}
For the lumped chain,
\begin{equation}\label{eq:lumped-collision-eigenfunction}
 \overline K(J-c)=\rho(J-c),
\end{equation}
and
\begin{equation}\label{eq:lumped-collision-lower}
 \overline d(t)
 \ge
 \frac{\theta(k-1)}{k\theta+1}\rho^t.
\end{equation}
\end{proposition}

\begin{proof}
Conditionally on $S$, the two coordinates move independently. Thus
\begin{align*}
 (K H_{r,s})(x)
 &=\mathbb P_x(X_1(r)=X_1(s))\\
 &=\mathbb E\langle S_{x_r,\bullet},S_{x_s,\bullet}\rangle
 =\frac1k+\rho H_{r,s}(x),
\end{align*}
where the last equality is \eqref{eq:row-inner-product}. Since
\[
 \frac{1/k}{1-\rho}
 =\frac{\theta+1}{k\theta+1}
 =c,
\]
and the kernel preserves constants, this proves \eqref{eq:collision-eigenfunction}. Iterating the eigenfunction identity and taking expectations gives \eqref{eq:exact-two-coordinate-correlation}.

Stationarity and \eqref{eq:collision-eigenfunction} give
\[
 \pi(F_{r,s})
 =\rho\pi(F_{r,s}).
\]
Since $\rho<1$, this expectation is zero, proving the stated stationary collision probability. Apply the event characterization of total variation to the event $A_{r,s}$. By \eqref{eq:exact-two-coordinate-correlation},
\begin{align*}
 \left\|K^t(x,\cdot)-\pi\right\|_{\TV}
 &\ge\left|K^t(x,A_{r,s})-\pi(A_{r,s})\right|\\
 &=\left|\mathbb P_x(X_t(r)=X_t(s))-c\right|
 =\rho^t|H_{r,s}(x)-c|.
\end{align*}
This proves \eqref{eq:pair-state-lower}. When $x_r=x_s$, its coefficient is
\[
 1-c=\frac{\theta(k-1)}{k\theta+1};
\]
when $x_r\ne x_s$, its coefficient is
\[
 c=\frac{\theta+1}{k\theta+1}.
\]
The first case gives \eqref{eq:constant-word-lower}. Both types of initial state exist because $n,k\ge2$, so maximizing over $x$ gives
\[
 d(t)\ge
 \frac{\max\{\theta+1,\theta(k-1)\}}{k\theta+1}\rho^t.
\]
Finally,
\begin{align*}
\bigl((J-c)\circ\hist\bigr)(x)
 &=J(\hist(x))-c =\binom{n}{2}^{-1}\sum_{1\le r<s\le n}H_{r,s}(x)-c\\
 &=\binom{n}{2}^{-1}\sum_{1\le r<s\le n}H_{r,s}(x)
 -\binom{n}{2}^{-1}\sum_{1\le r<s\le n}c =\binom{n}{2}^{-1}\sum_{1\le r<s\le n}F_{r,s}(x).
\end{align*}
Thus \eqref{eq:collision-eigenfunction} and \eqref{eq:histogram-intertwining} give \eqref{eq:lumped-collision-eigenfunction}. Since $\overline\pi(J)=c$ and $J(n,0,\ldots,0)=1$,
\[
 \overline d(t)
 \ge
 \left|(\overline K^tJ)(n,0,\ldots,0)-\overline\pi(J)\right|
 =(1-c)\rho^t
 =\frac{\theta(k-1)}{k\theta+1}\rho^t. \qedhere
\]
\end{proof}

\begin{remark}[Two-state collision factor]
For fixed distinct $r,s\in[n]$, the partition
\[
 \{A_{r,s},A_{r,s}^c\}
\]
is strongly lumpable. Indeed, if
\[
 Y_t:=\mathbf1_{\{X_t(r)=X_t(s)\}},
\]
then
\[
 \mathbb P_x(Y_1=1)
 =
 \frac1k+\rho H_{r,s}(x)
\]
depends on $x$ only through $Y_0=H_{r,s}(x)$. With the states ordered as ``equal'' and ``different,'' the resulting two-state transition matrix is
\[
 \widehat K=
 \begin{pmatrix}
  \frac1k+\rho & 1-\frac1k-\rho\\[3pt]
  \frac1k & 1-\frac1k
 \end{pmatrix}.
\]
Its eigenvalues are $1$ and $\rho$. Its stationary distribution is $(c,1-c)$, since
\[
 c=\frac1k+\rho c
 \qquad\text{and hence}\qquad
 c=\frac{\theta+1}{k\theta+1}.
\]
The function $f(y):=y-c$ is centered in the sense that it has mean zero under this stationary distribution. Moreover,
\[
 (\widehat Kf)(y)
 =\mathbb E[Y_{t+1}-c\mid Y_t=y]
 =\frac1k+\rho y-c
 =\rho(y-c)
 =\rho f(y),
\]
so $f$ is an eigenfunction with eigenvalue $\rho$. Lifting it to $[k]^n$ gives
\[
 f(H_{r,s}(x))=H_{r,s}(x)-c=F_{r,s}(x).
\]
Thus the collision eigenfunction arises from an exact two-state Markov factor of the full chain.
\end{remark}

The next theorem uses a trace argument to identify the second-largest eigenvalue.

\begin{theorem}
\label{thm:exact-second-eigenvalue}
Assume $n,k\ge2$. 
Let $K$ be the transition kernel of the twisted Bose--Einstein Markov chain on $[k]^n$. 
Then
\begin{equation}
 \lambda_2(K)=\lambda_2(\overline K)=\rho.
 \label{eq:exact-second-eigenvalue}
\end{equation}
\end{theorem}

\begin{proof}
The formula defining $Q_A$ makes sense for every $k\times k$ matrix $A$, and
\begin{align*}
 \operatorname{tr}(Q_A)
 &=\sum_{x\in[k]^n}Q_A(x,x)
 =\sum_{x\in[k]^n}\prod_{r=1}^n A_{x_rx_r}
 =\left(\sum_{a=1}^kA_{aa}\right)^n
 =(\operatorname{tr}A)^n.
\end{align*}
Consequently, \eqref{eq:matrix-product} gives
\begin{equation}
 \operatorname{tr}(K^t)
 =
 \mathbb E\operatorname{tr}(Q_{R_t})
 =
 \mathbb E(\operatorname{tr}R_t)^n.
 \label{eq:BE-trace-moment}
\end{equation}
Since each row of $S_i$ has mean $(1/k,\ldots,1/k)$,
\[
 \mathbb ES_i=\frac1kJ_k,
\]
where $J_k$ is the all-ones matrix. Independence therefore gives, for $t\ge1$,
\[
 \mathbb ER_t
 =(\mathbb ES_1)\cdots(\mathbb ES_t)
 =\left(\frac1kJ_k\right)^t
 =\frac1kJ_k.
\]
Hence
\begin{equation}
 \mathbb E\operatorname{tr}R_t
 =
 \operatorname{tr}\left(\frac1kJ_k\right)
 =1.
 \label{eq:BE-trace-mean}
\end{equation}
Define
\[
 Z_t:=\operatorname{tr}R_t-1.
\]
Then
\begin{equation}
 \mathbb EZ_t=0
 \qquad(t\ge1).
 \label{eq:BE-Z-mean}
\end{equation}
Moreover, $R_t$ is row-stochastic, so
\[
 0\le (R_t)_{aa}\le1
 \quad\Longrightarrow\quad
 0\le\operatorname{tr}R_t\le k.
\]
Since $k\ge2$, it follows that
\begin{equation}
 |Z_t|\le k-1.
 \label{eq:BE-Z-bound}
\end{equation}

We next use the two-coordinate chain. Let $K_2$ denote the same twisted Bose--Einstein kernel on $[k]^2$, with $k$ and $\theta$ unchanged. The law of $R_t$ depends on $k$ and $\theta$, but not on the word length $n$. Applying \eqref{eq:BE-trace-moment} with $n=2$ gives
\begin{equation}
 \mathbb E(\operatorname{tr}R_t)^2
 =
 \operatorname{tr}(K_2^t).
 \label{eq:second-trace-moment}
\end{equation}
Formula \eqref{eq:BE-kernel} gives
\[
 K_2(x,y)=
 \begin{cases}
  (\theta+k)/(k^2(\theta+1)),
  &x_1=x_2,\ y_1=y_2,\\
  \theta/(k^2(\theta+1)),
  &x_1=x_2,\ y_1\ne y_2,\\
  1/k^2,&x_1\ne x_2.
 \end{cases}
\]
The rows of $K_2$ depend only on whether $x_1=x_2$, so $\operatorname{rank}(K_2)\le2$. The constant function and Proposition~\ref{prop:collision-eigenfunction} give the distinct nonzero eigenvalues $1$ and $\rho$, so $\operatorname{rank}(K_2)\ge2$. Thus $K_2$ has rank two and these are its only nonzero eigenvalues. Consequently,
\begin{equation}
 \operatorname{tr}(K_2^t)
 =1+\rho^t
 \qquad(t\ge1).
 \label{eq:BE-two-coordinate-trace}
\end{equation}
Applying \eqref{eq:BE-trace-mean}, \eqref{eq:second-trace-moment}, and \eqref{eq:BE-two-coordinate-trace} yields
\begin{align}
 \mathbb EZ_t^2
 &=\mathbb E(\operatorname{tr}R_t-1)^2
 =\mathbb E(\operatorname{tr}R_t)^2
 -2\mathbb E\operatorname{tr}R_t+1 \notag\\
 &=\operatorname{tr}(K_2^t)-1
 =\rho^t
 \qquad(t\ge1).
 \label{eq:BE-Z-second-moment}
\end{align}

For $t\ge1$, \eqref{eq:BE-trace-moment}, \eqref{eq:BE-Z-mean}, \eqref{eq:BE-Z-bound}, and \eqref{eq:BE-Z-second-moment} give
\begin{align*}
 |\operatorname{tr}(K^t)-1|
 &=\left|\mathbb E(1+Z_t)^n-1\right|
 =\left|\sum_{r=2}^n\binom nr\mathbb EZ_t^r\right|\\
 &\le\sum_{r=2}^n\binom nr\mathbb E|Z_t|^r
 =
 \sum_{r=2}^n\binom nr
 \mathbb E\!\left[|Z_t|^{r-2}Z_t^2\right] \le\sum_{r=2}^n\binom nr(k-1)^{r-2}\mathbb EZ_t^2\\
 &=\rho^t\sum_{r=2}^n\binom nr(k-1)^{r-2}
 =C\rho^t,
\end{align*}
where
\[
 C:=\sum_{r=2}^n\binom nr(k-1)^{r-2}.
\]
The general twisted Burnside calculation in the introduction shows that the chain is reversible and has no negative eigenvalues. Moreover, \eqref{eq:BE-kernel} is strictly positive, so the eigenvalue $1$ is simple. Write $\lambda_j:=\lambda_j(K)$. For every $j\ge2$ and $m\ge1$,
\[
 |\lambda_j|^{2m}
 \le\sum_{\ell=2}^{k^n}\lambda_\ell^{2m}
 =\operatorname{tr}(K^{2m})-1
 \le C\rho^{2m}.
\]
Taking $(2m)$th roots and letting $m\to\infty$ yields
\[
 |\lambda_j|\le\rho
 \qquad(j\ge2).
\]
Finally, Proposition~\ref{prop:collision-eigenfunction} shows that $\rho$ is attained by $K$. If $\overline Kh=\lambda h$, then \eqref{eq:histogram-intertwining} gives $K(h\circ\hist)=\lambda(h\circ\hist)$, so every eigenvalue of $\overline K$ is an eigenvalue of $K$. On the other hand, $J(n,0,\ldots,0)=1\ne(n-2)/n=J(n-1,1,0,\ldots,0)$, so \eqref{eq:lumped-collision-eigenfunction} shows that $\rho$ is attained by $\overline K$. This proves \eqref{eq:exact-second-eigenvalue}.
\end{proof}

For $n\ge2$, since $c\in(0,1)$ and $H_{r,s}(x)\in\{0,1\}$, we have $F_{r,s}(x)\ne0$ for every $x\in[k]^n$. Hence \eqref{eq:pair-state-lower} and the standard reversible-chain spectral bound~\cite[Chapter~12]{LPW09} imply
\[
 \lim_{t\to\infty}
 \left\|K^t(x,\cdot)-\pi\right\|_{\TV}^{1/t}
 =\rho
 \qquad\text{for every }x\in[k]^n.
\]

In particular, the labeled and lumped chains have relaxation time $t_{\rel}(K)=t_{\rel}(\overline K)=k(\theta+1)/(k\theta+1)$; for the corresponding parking-function chains it equals $2(n+1)/(n+2)$ and converges to $2$. Thus, for fixed $\theta>0$, the relaxation time remains bounded even though, in the regimes above, worst-case total-variation mixing takes order $\log n$: the logarithmic scale is not caused by a closing spectral gap.

\begin{remark}[Relation to earlier spectral results]
For $k=2$, Diaconis--Zhong~\cite[Theorem~3]{DZ21} diagonalized the orbit-lumped twisted chain and obtained the first nontrivial eigenvalue $1/(2(\theta+1))=\rho$; they also proved matching constant-start total-variation decay for $\theta\ge1$ in \cite[Theorem~4]{DZ21}. For $\theta=1$, Diaconis--Lin--Ram~\cite[Theorem~1.2]{DLR25a+} later diagonalized the full binary chain. They also explain in \cite[Section~6.2]{DLR25a+} why a full diagonalization becomes substantially harder for $k>2$. Theorem~\ref{thm:exact-second-eigenvalue} identifies the second-largest eigenvalue of both the labeled and orbit-lumped chains for every $k\ge2$ and $\theta>0$.
\end{remark}

\subsection{Total-Variation Upper Bounds}

In what follows, let $K$ and $\overline{K}$ be the transition kernels of the twisted Bose--Einstein Markov chain and the lumped twisted Bose--Einstein chain. Let $d(t)$ and $t_{\mix}(\epsilon)$ (respectively, $\overline{d}(t)$ and $\overline{t}_{\mix}(\epsilon)$) be the distance function and the mixing time of $K$ (respectively, $\overline{K}$). 

\begin{theorem}[Coordinate-chain bound]
\label{thm:coordinate-upper}
Assume $k\ge2$. For every $t\ge0$,
\begin{equation}\label{eq:coordinate-upper}
 \overline d(t)
 \le d(t)
 \le
 \frac n2\sqrt{2k}\,\rho^{t/2}.
\end{equation}
Consequently, for $0<\varepsilon<1$,
\begin{equation}\label{eq:two-parameter-upper}
 \overline t_{\mix}(\varepsilon)
 \le t_{\mix}(\varepsilon)
 \le
 \left\lceil
 \frac{2\log\!\left(n\sqrt{k/2}/\varepsilon\right)}
 {\log\!\left(k(\theta+1)/(k-1)\right)}
 \right\rceil.
\end{equation}
\end{theorem}

\begin{proof}
For a generic initial state $u\in[k]^n$, the conditional product representation \eqref{eq:conditional-product} gives
\[
 \mathcal L(X_t\mid R_t,X_0=u)
 =\bigotimes_{i=1}^{n}(e_{u_i}R_t).
\]
Therefore, for $x,y\in[k]^n$, convexity of total variation, the product-measure inequality \eqref{eq:TV-product}, $\|w\|_1\le\sqrt{k}\|w\|_2$, Cauchy--Schwarz, and Lemma~\ref{lem:contraction} give
\begin{align*}
 \left\|K^t(x,\cdot)-K^t(y,\cdot)\right\|_{\TV}
 &\le\sum_{i=1}^n
 \mathbb E\left\|e_{x_i}R_t-e_{y_i}R_t\right\|_{\TV}\\
 &\le\frac{\sqrt{k}}2\sum_{i=1}^n
 \sqrt{\mathbb E\left\|(e_{x_i}-e_{y_i})R_t\right\|_2^2}\\
 &=\frac12\sum_{i=1}^n
 \sqrt{k\rho^t\left\|e_{x_i}-e_{y_i}\right\|_2^2}\\
 &\le\frac n2\sqrt{2k}\,\rho^{t/2}.
\end{align*}
Since $\pi$ is stationary, another application of convexity gives
\begin{align*}
 \left\|K^t(x,\cdot)-\pi\right\|_{\TV}
 &\le\sum_{y\in[k]^n}\pi(y)
 \left\|K^t(x,\cdot)-K^t(y,\cdot)\right\|_{\TV}\\
 &\le\frac n2\sqrt{2k}\,\rho^{t/2}.
\end{align*}
Taking the maximum over $x$ and using \eqref{eq:histogram-TV-contraction} proves \eqref{eq:coordinate-upper}. Finally, $\rho^{-1}=k(\theta+1)/(k-1)$; solving the upper bound for $t$ proves \eqref{eq:two-parameter-upper}.
\end{proof}

\begin{corollary}[Growing number of labels]
\label{cor:polynomial-upper}
Assume $k\ge2$. For fixed $\theta>0$ and $0<\varepsilon<1$,
\begin{equation}\label{eq:polynomial-upper}
 \overline t_{\mix}(\varepsilon)
 \le t_{\mix}(\varepsilon)
 \le
 \left\lceil
 2\log_{\theta+1}n+\log_{\theta+1}k
 +2\log_{\theta+1}(1/\varepsilon)-\log_{\theta+1}2
 \right\rceil.
\end{equation}
Consequently, if $k_n\le n^C$ for all sufficiently large $n$ and some fixed $C$, then
\[
 t_{\mix}(\varepsilon),\ \overline t_{\mix}(\varepsilon)
 =O_{\theta,\varepsilon,C}(\log n).
\]
More precisely, if $k_n=n^{\beta+o(1)}$ for some fixed $\beta\ge0$, then
\begin{equation}\label{eq:polynomial-colors}
 \overline t_{\mix}(\varepsilon)
 \le t_{\mix}(\varepsilon)
 \le(2+\beta)\log_{\theta+1}n+o(\log n).
\end{equation}
\end{corollary}

\begin{proof}
Since
\[
 \log\!\left(\frac{k(\theta+1)}{k-1}\right)
 \ge\log(\theta+1),
\]
Theorem~\ref{thm:coordinate-upper} gives \eqref{eq:polynomial-upper}. The remaining claims follow from $\log k_n\le C\log n$ eventually or $\log_{\theta+1}k_n=\beta\log_{\theta+1}n+o(\log n)$, respectively.
\end{proof}

\subsection{Entropy Lower Bounds}

Set $\mu:=\psi(\theta+1)-\psi(1)$ and $m_*:=\min\{n,k\}$.

The proof of the lower bound compares the slowly growing expected empirical entropy of the chain with its nearly maximal stationary value; normalized empirical entropy then serves as a test function for total variation.

For a probability vector $q=(q_1,\ldots,q_k)$, define
\[
 H(q):=-\sum_{a=1}^kq_a\log q_a,
\]
with $0\log0:=0$. For $x\in[k]^n$, set
\[
 \widehat p(x):=\frac1n(i_1(x),\ldots,i_k(x)),
 \qquad
 \widehat H(x):=H(\widehat p(x)).
\]
For $\bm{i}\in\Comp_{n,k}$, set
\[
 \overline H(\bm{i}):=H(\bm{i}/n),
 \qquad
 \widehat H=\overline H\circ\hist.
\]

\begin{theorem}[Entropy lower bound]\label{thm:growing-lower}
If $m_*\ge2$, then, for every $t\ge0$,
\begin{equation}\label{eq:growing-lower}
 d(t)
 \ge\overline d(t)
 \ge
 1-\frac{t\mu+\log(2+1/\theta)}{\log m_*}.
\end{equation}
Consequently, for every $0<\varepsilon<1$,
\begin{equation}\label{eq:growing-finite-mixing-lower}
 t_{\mix}(\varepsilon)
 \ge\overline t_{\mix}(\varepsilon)
 \ge
 \frac{(1-\varepsilon)\log m_*-\log(2+1/\theta)}
 {\mu}.
\end{equation}
\end{theorem}

\begin{proof}
Every $\bm{i}\in\Comp_{n,k}$ has at most $m_*$ nonzero entries, so
\begin{equation}\label{eq:empirical-entropy-range}
 0\le\overline H(\bm{i})\le\log m_*.
\end{equation}

Let
\[
 x_\star:=(1,\ldots,1),
 \qquad
 \bm{i}_\star:=\hist(x_\star)=(n,0,\ldots,0).
\]
Since $\overline\pi$ is the pushforward of $\pi$ under $\hist$ and $\widehat H=\overline H\circ\hist$, \eqref{eq:histogram-intertwining} and its iterates give
\[
 \overline\pi(\overline H)
 =\pi(\widehat H)
 =\mathbb E_\pi\widehat H,
 \qquad
 (\overline K^t\overline H)(\bm{i}_\star)
 =(K^t\widehat H)(x_\star)
 =\mathbb E_{x_\star}\widehat H(X_t).
\]
The contraction bound \eqref{eq:histogram-TV-contraction} and the bounded-test-function characterization of total variation~\cite[Proposition~4.5]{LPW09}, applied to $\overline H/\log m_*$, therefore give
\begin{align}
 d(t)
 &\ge\overline d(t)
 \ge\left\|\overline K^t(\bm{i}_\star,\cdot)-\overline\pi\right\|_{\TV}\notag\\
 &\ge
 \frac{
  \overline\pi(\overline H)
  -(\overline K^t\overline H)(\bm{i}_\star)
 }{\log m_*}\notag\\
 &=
 \frac{
  \mathbb E_\pi\widehat H
  -\mathbb E_{x_\star}\widehat H(X_t)
 }{\log m_*}.
 \label{eq:entropy-TV-test}
\end{align}
Thus it remains to upper-bound the expected entropy at time $t$ from $x_\star$ and to lower-bound its stationary expectation.

Put
\[
 q_s:=e_1S_1\cdots S_s,
 \qquad q_0=e_1.
\]
If $Z\sim\Dirichlet(\theta/k,\ldots,\theta/k)$, differentiating the identity
\[
 \mathbb EZ_1^u
 =\frac{\Gamma(\theta)\Gamma(\theta/k+u)}
 {\Gamma(\theta/k)\Gamma(\theta+u)}
\]
at $u=1$ gives
\[
 \mathbb E[Z_1\log Z_1]
 =\frac1k\left[
 \psi\left(1+\frac\theta k\right)-\psi(\theta+1)
 \right].
\]
Hence
\begin{equation}\label{eq:row-entropy}
 h:=\mathbb EH(Z)
 =\psi(\theta+1)-\psi\left(1+\frac\theta k\right)
 \le\mu.
\end{equation}

The standard upper bound for the entropy of a mixture~\cite[Equation~(2)]{KT17}, applied with mixing weights $q$ and component distributions $B_{a,\bullet}$, gives
\begin{equation}\label{eq:entropy-mixture}
 H(qB)\le H(q)+\sum_{a=1}^kq_aH(B_{a,\bullet})
\end{equation}
for every probability vector $q$ and every row-stochastic matrix $B$. Since $q_s=q_{s-1}S_s$, while $q_{s-1}$ is independent of $S_s$ and all rows of $S_s$ have the same law, \eqref{eq:entropy-mixture} gives
\[
 \mathbb EH(q_s)\le\mathbb EH(q_{s-1})+h.
\]
Iterating from $H(q_0)=0$ yields
\begin{equation}\label{eq:entropy-growth}
 \mathbb EH(q_t)\le t h\le t\mu.
\end{equation}

By \eqref{eq:conditional-product}, conditionally on $q_t$ the coordinates of $X_t$, started from $x_\star$, are independent with common law $q_t$. Thus
\[
 \mathbb E_{x_\star}[\widehat p(X_t)\mid q_t]=q_t.
\]
Concavity of entropy and \eqref{eq:entropy-growth} imply
\begin{equation}\label{eq:dynamic-entropy-upper}
 (\overline K^t\overline H)(\bm{i}_\star)
 =\mathbb E_{x_\star}\widehat H(X_t)
 \le\mathbb EH(q_t)\le t\mu.
\end{equation}

Each label appearing $i_a(x)$ times contributes $i_a(x)$ copies of $\log i_a(x)$ to the sum below, so
\begin{equation}\label{eq:entropy-class-identity}
 \widehat H(x)
 =\log n-\frac1n\sum_{r=1}^n\log i_{x_r}(x).
\end{equation}
The stationary law \eqref{eq:BE-pi} is exchangeable. Since
\[
 i_{x_1}(x)=1+\sum_{r=2}^n\mathbf1_{\{x_r=x_1\}},
\]
Proposition~\ref{prop:collision-eigenfunction} gives
\begin{align*}
 \mathbb E_{\pi}i_{x_1}(x)
 &=1+(n-1)\pi(x_1=x_2)\\
 &=1+(n-1)\frac{\theta+1}{k\theta+1}
 \le\frac{n}{m_*}\left(2+\frac1\theta\right).
\end{align*}
The last inequality follows by considering separately $k\le n$ and $k\ge n$. Exchangeability in \eqref{eq:entropy-class-identity} and Jensen's inequality therefore give
\begin{align}
 \overline\pi(\overline H)
 =\mathbb E_{\pi}\widehat H
 &=\log n-\mathbb E_\pi\log i_{x_1}(x)\notag\\
 &\ge\log n-\log\mathbb E_\pi i_{x_1}(x)
 \ge\log m_*-\log\left(2+\frac1\theta\right).
 \label{eq:stationary-entropy-lower}
\end{align}

Substituting \eqref{eq:dynamic-entropy-upper} and \eqref{eq:stationary-entropy-lower} into \eqref{eq:entropy-TV-test} proves \eqref{eq:growing-lower}.

Taking $t=\overline t_{\mix}(\varepsilon)$ and rearranging the lower bound for $\overline d(t)$ proves the second inequality in \eqref{eq:growing-finite-mixing-lower}; the first follows from $d(t)\ge\overline d(t)$.
\end{proof}

\begin{proof}[Proof of Theorem~\ref{BEbound}]
The upper bound, collision lower bound, spectral identity, and entropy lower bound follow from Theorem~\ref{thm:coordinate-upper}, Proposition~\ref{prop:collision-eigenfunction}, Theorem~\ref{thm:exact-second-eigenvalue}, and Theorem~\ref{thm:growing-lower}, respectively. If $k_n=n^{\beta+o(1)}$ with $\beta>0$, then $\log\min\{n,k_n\}=\Theta(\log n)$, so \eqref{eq:growing-finite-mixing-lower} and Corollary~\ref{cor:polynomial-upper} give $t_{\mix}(\epsilon)=\Theta_{\theta,\beta,\epsilon}(\log n)$.
\end{proof}

\subsection{Consequences for Parking Functions}

In this section, we apply our results from the previous section to the Burnside process on $\PF_n$. The following two corollaries provide upper and lower bounds on the mixing time. 

\begin{corollary}[Parking-function second eigenvalue and upper bound]
\label{thm:main-upper}
For $n\ge2$,
\begin{equation}
 \lambda_{2,\PF} = \lambda_2(K_{\PF})
 =\frac{n}{2(n+1)}.
 \label{eq:PF-spectrum}
\end{equation}
For every $t\ge1$,
\begin{equation}
 d_{K_{\PF}}(t)
 \le
 n\sqrt{\frac{n+1}{2}}\lambda_{2,\PF}^{t/2}.
 \label{eq:main-upper}
\end{equation}
Finally, for every fixed $0<\varepsilon<1$,
\begin{equation}\label{eq:PF-mixing-upper}
 t_{\mix,K_{\PF}}(\varepsilon)
 \le3\log_2n+O_\varepsilon(1).
\end{equation}
\end{corollary}

\begin{proof}
Set $\theta=1$ and $k=n+1$ in Theorems~\ref{thm:exact-second-eigenvalue} and~\ref{thm:coordinate-upper}, and apply Theorem~\ref{thm:BE-PF-transfer}.
\end{proof}

By Proposition~\ref{prop:collision-eigenfunction} and the two bijections, for distinct $r,s\in[n]$ the corresponding $\lambda_{2,\PF}$-eigenfunctions are $u\mapsto\mathbf1_{\{u_r=u_s\}}-2/(n+2)$ on $\PF_n$ and $L\mapsto\mathbf1_{\{r,s\text{ lie in the same vertical run of }L\}}-2/(n+2)$ on $\LcDc_n$.

\begin{corollary}[Parking-function lower bound]\label{thm:PF-log-lower}
For $n\ge2$ and $t\ge1$,
\begin{equation}\label{eq:PF-TV-lower}
 d_{K_{\PF}}(t)
 \ge
 \max\left\{
  \frac{n}{n+2}\lambda_{2,\PF}^t,\,
  1-\frac{t+\log 3}{\log n}
 \right\}.
\end{equation}
Consequently, for every fixed $0<\varepsilon<1$,
\begin{equation}\label{eq:PF-mixing-lower}
 t_{\mix,K_{\PF}}(\varepsilon)
 \ge(1-\varepsilon)\log n-\log 3
\end{equation}
for all sufficiently large $n$.
\end{corollary}

\begin{proof}
Set $\theta=1$ and $k=n+1$ in Proposition~\ref{prop:collision-eigenfunction} and Theorem~\ref{thm:growing-lower}, and apply Theorem~\ref{thm:BE-PF-transfer}.
\end{proof}

\begin{proof}[Proof of Theorem~\ref{mainthm}]
Combine Corollaries~\ref{thm:main-upper} and~\ref{thm:PF-log-lower}.
\end{proof}

We also have the following mixing time result for the lumped process on $\IPF_n$. 

\begin{corollary}[Lumped parking-function chain]\label{cor:lumped-PF-bounds}
For $n\ge2$,
\[
 \lambda_2(\overline K_{\PF})
 =\frac{n}{2(n+1)}
 =\lambda_{2,\PF}.
\]
Moreover, for every $t\ge1$,
\[
 \max\left\{
  \frac{n}{n+2}\lambda_{2,\PF}^t,\,
  1-\frac{t+\log3}{\log n}
 \right\}
 \le d_{\IPF}(t)
 \le n\sqrt{\frac{n+1}{2}}\lambda_{2,\PF}^{t/2}.
\]
Consequently, for every fixed $0<\varepsilon<1$,
\[
 t_{\mix,\IPF}(\varepsilon)=\Theta_\varepsilon(\log n).
\]
\end{corollary}

\begin{proof}
Set $\theta=1$ and $k=n+1$ in Proposition~\ref{prop:collision-eigenfunction} and Theorems~\ref{thm:exact-second-eigenvalue}, \ref{thm:coordinate-upper}, and~\ref{thm:growing-lower}, and apply Theorem~\ref{thm:BE-PF-transfer}.
\end{proof}

Under the induced bijection $\IPF_n\to\Dc_n$, the same conclusions hold for the lumped Burnside process on $\Dc_n$.

\Address

\end{document}